\documentclass{amsart}%
\usepackage{amssymb, amsmath,  amsfonts}%
\usepackage{amsmath}%
\usepackage{amsfonts}%
\usepackage{amssymb}%
\usepackage{graphicx}
\newtheorem{thm}{Theorem}[section]
\newtheorem{prop}[thm]{Proposition}
\newtheorem{lemma}[thm]{Lemma}
\newtheorem{cor}[thm]{Corollary}
\newtheorem{defn}[thm]{Definition}
\newtheorem{example}[thm]{Example}
\newtheorem{remark}[thm]{Remark}

\numberwithin{equation}{section}

\begin{document}
\title{Quantum Pieri rules for tautological subbundles}
\author{Naichung Conan Leung}
\address{The Institute of Mathematical Sciences and Department of Mathematics, The
Chinese University of Hong Kong, Shatin, Hong Kong}
\email{leung@math.cuhk.edu.hk}
\author{Changzheng Li}
\address{Kavli Institute for the Physics and Mathematics of the Universe (WPI),
  The University of Tokyo,
5-1-5 Kashiwa-no-Ha,Kashiwa City, Chiba 277-8583, Japan}
\email{changzheng.li@ipmu.jp}
\date{}

\begin{abstract}
We give quantum Pieri rules for quantum cohomology of Grassmannians of
classical types, expressing the quantum product of Chern classes of 
the tautological subbundles with general cohomology classes. We derive them
by showing the relevant genus zero, three-pointed Gromov-Witten invariants
coincide with certain classical intersection numbers.

\end{abstract}
\maketitle

{\allowdisplaybreaks[4] }










\section{Introduction}

The complex Grassmannian $Gr(k,n+1)$ parameterizes $k$-dimensional complex
vector subspaces of $\mathbb{C}^{n+1}$. It can be written as $X=G/P$ with $G$
being a complex Lie group of type $A$, i.e. $G=SL(n+1,\mathbb{C})$, and $P$ being
a maximal parabolic subgroup of $G$. We will continue to call such $X$'s as
\textit{Grassmannians} even when $G$ is not of type $A$. Indeed when $G$ is a
classical Lie group of type $B,C$ or $D$, such a Grassmannian parameterizes
subspaces in a vector space which are isotropic with respect to a
non-degenerate skew-symmetric or symmetric bilinear form. Therefore it is
usually called an \textit{isotropic Grassmannian}. Recall that the
tautological subbundle $\mathcal{S}$ over any point $[V]\in Gr(k,n+1)$ is just
the $k$-dimensional vector subspace $V$ itself. And it restricts to the
tautological subbundle $\mathcal{S}$ of any isotropic Grassmannian.

The cohomology ring $H^{\ast}(X,\mathbb{Z})$ of an isotropic Grassmannian
$X=G/P$, or more generally a generalized flag variety, has a natural basis
consisting of Schubert cohomology classes $\sigma^{u}$, labelled by a subset of
the Weyl group $W$ of $G$. The (small) quantum cohomology ring $QH^{\ast}(X)$
of $X$, as a vector space, is isomorphic to $H^{\ast}(X)\otimes\mathbb{Q}[t]$.
The quantum ring structure is a deformation of the ring structure on $H^{\ast
}(X)$ by incorporating three-pointed, genus zero Gromov-Witten invariants of
$X$. Since $H_{2}\left(  X,\mathbb{Z}\right)  \cong\mathbb{Z}$, the homology
class of a holomorphic curve in $X$ is labelled by its degree $d$.
In the case of $X=IG(k, 2n)$ being a Grassmannian of type $C_n$, the Schubert cohomology
classes $\sigma^{u}=\sigma^{\mathbf{a}}$ can also be labelled by
\textit{shapes} $\mathbf{a}$, which are certain pairs of partitions.
Every nonzero Chern class   $c_{p}(\mathcal{S}^{*})=(-1)^pc_p(\mathcal{S})=\sigma^{p}$ (up to a scale factor of $2$) is then a special Schubert class given by
a special shape $p$, and they generate the quantum cohomology ring $QH^*(IG(k, 2n))$. One of the main results of the present paper is the following
 formula.

\noindent\textbf{Quantum Pieri Rule for tautological subbundles of $IG(k, 2n)$}(Theorem \ref{thmQPRforIGIG222}) {\itshape
For any shape $\mathbf{a}$ and every special shape $p$,  in $QH^*(IG(k, 2n))$, we have  \[
\sigma^{p}\star\sigma^{\mathbf{a}}=\sum2^{e(\mathbf{a},\mathbf{b})}%
\sigma^{\mathbf{b}}+t\sum2^{e(\tilde{\mathbf{a}},\tilde{\mathbf{c}})}%
\sigma^{\mathbf{c}}.
\]
}

\noindent Here $\tilde{\mathbf{a}}$ and $\tilde{\mathbf{c}}$ are shapes
associated to $\mathbf{a}$ and $\mathbf{c}$ respectively;  $e(\mathbf{a},
\mathbf{b})$ and ${e(\tilde{\mathbf{a}}, \tilde{\mathbf{c}})}$ are
cardinalities of certain combinatorial sets,
determined by  the classical Pieri rules of Pragacz and Ratajski \cite{pragrat}. We have also obtained similar formulas for Grassmanianns of type $B$ and $D$, details of which are given in section \ref{sectionreformalation}.

The aforementioned \textit{quantum Pieri rule}  is a quantum version of the classical Pieri rule    for isotropic
Grassmannians. The famous classical Pieri rules are known firstly for complex
Grassmannians (see e.g. \cite{gri}). For $X=Gr(k, n+1)$, they describe the cup
product of a general Schubert class in $H^{*}(X)$ with  $c_{p}(\mathcal{S}%
^{*})$ or $c_{p}(\mathcal{Q})$, where $\mathcal{Q}$ is the tautological
quotient bundle over $X$ given by the exact sequence  $0\rightarrow
\mathcal{S}\rightarrow\mathbb{C}^{n+1}\rightarrow\mathcal{Q}\rightarrow0$.
It was generalized for other partial flag varieties of type $A$, firstly given by Lascoux and Sch\"utzenberger \cite{LasSch},
and was also   generalized for  Grassmannians $X$ of type $B, C$ or $D$.
Note that there is also a tautological quotient bundle $\mathcal{Q}$ over $X$.  When $X$ parameterizes maximal isotropic subspaces,
 (roughly speaking) there is no difference between the Chern classes
   of $\mathcal{S}^*$ and $\mathcal{Q}$, and  the classical
  Pieri rules  has been  given by Hiller-Boe \cite{HillerBoe}.
When     $X$
parameterizes non-maximal isotropic subspaces, the classical Pieri rules with
respect to $c_{p}(\mathcal{S}^{*})$ have been given by  Pragacz and Ratajski
(\cite{pragrat}, \cite{pragratEvenorth}), while  the classical Pieri rules with
respect to $c_{p}(\mathcal{Q})$ are just covered in the recent work of Buch, Kresch and Tamvakis \cite{BKT2} on quantum Pieri rules.
In contrast to complex Grassmannians, knowing either of them cannot deduce the other one.
There is also a previous work of
Sert\"oz \cite{Sertoz}  as well as a generalized classical Pieri rule given by
Bergeron and Sottile \cite{bergSottile}, which gives the formula for multiplying a Schubert class
on a complete flag variety of type $B$ or $C$ by a special Schubert class
pulled back from the Grassmanian of maximal isotropic subspaces.

The story of quantum Pieri rules are almost parallel to the story of the classical Pieri rules.
 The quantum  Pieri rules are also known firstly for complex
Grassmannians, which were firstly  given by   Bertram \cite{bert}.
They  were  generalized by Ciocan-Fontanine \cite{CFon} for other partial flag varieties of type $A$,
and by Kresch-Tamvakis (\cite{KreschTamvakisLagran}, \cite{KreschTamvakisortho}) for those $X$ that parameterize maximal isotropic subspaces.
Recently in \cite{BKT2},   Buch, Kresch and Tamvakis have given  us the quantum Pieri rules with
respect to $c_{p}(\mathcal{Q})$ for   those $X$ that parameterize non-maximal isotropic subspaces.
 In contrast to complex Grassmannians, (in general) the quantum
Pieri rules with respect to $c_{p}(\mathcal{Q})$ do not imply the
quantum Pieri rules with respect to  $c_{p}(\mathcal{S}^{*})$ and vice versa.

Our quantum Pieri rules are consequences of the following main technical result.

  \bigskip

\noindent\textbf{Main Theorem. }{\itshape Let $X=G/P$ be a Grassmannian of
type $B,C$ or $D$, and
$\mathcal{S}$ denote the tautological subbundle over $X$. Let $\sigma
^{u},\sigma^{v}$ be Schubert classes in $QH^{\ast}(X)$ with $\sigma^{u}%
=(-1)^pc_{p}(\mathcal{S})$ (possibly up to a scale factor of ${\frac{1}{2}}$, see section \ref{subsectisotroGrassm} for more details)
for any $p$. In the quantum product
\[
\sigma^{u}\star\sigma^{v}=\sigma^{u}\cup\sigma^{v}+\sum_{d\geq1}N_{u,v}%
^{w,d}t^{d}\sigma^{w},
\]
all the degree }$d$ {\itshape Gromov-Witten invariants $N_{u,v}^{w,d}$ coincide with
certain classical intersection numbers. More precisely, we have

\begin{enumerate}

\item If $d=1$, then there exist $u_{1}, v_{1}, w_{1}\in W$ such that $N_{u,
v}^{w, 1}=N_{u_{1}, v_{1}}^{w_{1}, 0}$.

\item If $d=2$, then there exist $u_{2}, v_{2}, w_{2}\in W$ such that $N_{u,
v}^{w, 2}=N_{u_{2}, v_{2}}^{w_{2}, 0}$.

\item If $d\geq3$, then $N_{u, v}^{w, d}=0$.
\end{enumerate}
}
\noindent Here $N_{u_{i},v_{i}}^{w_{i},0}$'s are classical intersection
numbers of the corresponding Schubert varieties in the complete flag variety
$G/B$, where $B\subset P$ is a Borel subgroup of $G$. The elements
$u_{i},v_{i},w_{i}$ can be explicitly written down in terms of $u,v,w$ as
given in Theorem \ref{thmQPRforIGIG} and Theorem \ref{thmQPRforOGOG}. In fact
$N_{u,v}^{w,2}$ also vanishes for some cases.

The above theorem  is an     application of the main results  of \cite{leungliQtoC}, where
the authors studied the \textquotedblleft quantum to classical" principle for
flag varieties of general type.
  Roughly speaking, the \textquotedblleft quantum to
classical" principle says that certain three-pointed genus zero Gromov-Witten
invariants are classical intersection numbers. Such phenomenon,
   probably for the first time, occurred in the proof
of quantum Pieri rule for partial flag varieties of type $A$ by
Ciocan-Fontanine \cite{CFon}, and later occurred in the elementary proof of
quantum Pieri rule for complex Grassmannians by Buch \cite{buch} and the work \cite{KreschTamvakisLagran}, \cite{KreschTamvakisortho} of Kresch and Tamvakis
 on Lagrangian and orthogonal Grassmannians.
This principle has been studied mainly for Grassmannians  in the works (especially) by  Buch-Kresch-Tamvakis
  (\cite{BKT1}, \cite{BKT2}), by Chaput-Manivel-Perrin (\cite{cmn}, \cite{chper}) and  by    Buch-Mihalcea (\cite{buchMihalcea},
   \cite{buchMihalcea22}).
There are relevant studies for some other cases   by Coskun \cite{coskun33} and by Li-Mihalcea  \cite{liMihal}.

The proofs of our quantum Pieri rules  are
combinatorial in nature. The ideas of all these proofs are the same,  namely we obtain all the theorems by showing that
the relevant Gromov-Witten invariants of degree $d$ vanish unless
$d$ is small enough (for instance $d\leq2$), and for such a small $d$ they
coincide with certain classical intersection numbers. Moreover, all these relevant classical intersection numbers   are exactly or can be calculated from certain structure constants
 in classical Pieri rules of same type.
We should note that in \cite{leungli33}, the authors  established   natural  filtered algebra structures
on $QH^{\ast}\left(G/B\right)$. Using     structures of
these filtrations, we    obtained relationships among three-pointed genus zero Gromov-Witten invariants for $G/B$ in \cite{leungliQtoC},
which enable us to carry out the above ideas in real proofs.
 Finally, we should also note that our quantum Pieri rules for type $B, D$ are not quite satisfying, as signs are involved in some cases.

This  paper is organized as follows. In section 2, we fix the notations and
review the main results of \cite{leungliQtoC}.  In section 3, we
reduce all the relevant Gromov-Witten invariants to certain classical intersection numbers, for   the
quantum Pieri rules for Grassmannians of type $B, C$ or $D$ with respect to
$c_{p}(\mathcal{S}^{*})$. In section 4, we obtain the quantum Pieri rules by computing those classical intersection
numbers in section 3 combinatorially. Finally in the appendix, we   reprove the well-known quantum Pieri rules for
Grassmannians of type $A$ (i.e. complex Grassmannians).

\section{Preliminary results}

\subsection{Notations}

\label{prelimiar}

We recall notations in \cite{leungliQtoC} which we will use here as well.
Readers can refer to section 2.1 of \cite{leungliQtoC} and references therein
for more details.

Let $G$ be a simply-connected complex simple Lie group of rank $n$,  $B\subset
G$ be a Borel subgroup and $\mathfrak{h}$ be the corresponding Cartan
subalgebra.  Let $\Delta=\{\alpha_{1}, \cdots, \alpha_{n}\}\subset
\mathfrak{h}^{*}$ be the simple  roots with the associated Dynkin diagram
$Dyn(\Delta)$ being the same as in section 11.4 of  \cite{hum}.  Let
$\{\alpha_{1}^{\vee}, \cdots, \alpha_{n}^{\vee}\}\subset\mathfrak{h}$ be the
simple  coroots, $\{\chi_{1}, \cdots, \chi_{n}\}$ be the fundamental weights
and  $R^{+}$ be the set of positive roots in the root system $R$.  Denote
$Q^{\vee}=\bigoplus_{i=1}^{n}\mathbb{Z}\alpha_{i}^{\vee}$ and $\rho=\sum
_{i=1}^{n}\chi_{i}$.  The Weyl group $W$ is generated by  the simple
reflections $s_{i}$'s on $\mathfrak{h}^{*}$  defined by $s_{i}(\beta
)=s_{\alpha_{i}}(\beta):=\beta-\langle\beta, \alpha_{i}^{\vee}\rangle
\alpha_{i}$ for each $i$, where  $\langle\cdot, \cdot\rangle:\mathfrak{h}%
^{*}\times\mathfrak{h}\rightarrow\mathbb{C}$ is the natural pairing.  Each
parabolic subgroup $P\supset B$ is in one-to-one correspondence with  a subset
$\Delta_{P}\subset\Delta$.  Let $\ell: W\rightarrow\mathbb{Z}_{\geq0}$ be the
length function, $W_{P}$ denote  the Weyl subgroup generated by $\{s_{\alpha
}~|~ \alpha\in\Delta_{P}\}$ and  $W^{P}$ denote the minimal length
representatives of the cosets $W/W_{P}$. Let $\omega_{P}$ denote the (unique)
longest element in $W_{P}$.

The (co)homology of a (generalized) flag variety $X=G/P$ is torsion free and
it has an additive basis of Schubert (co)homology classes  $\sigma_{u}$'s
(resp. $\sigma^{u}$'s) indexed by $W^{P}$.  Note that $\sigma^{u}\in
H^{2\ell(u)}(X, \mathbb{Z})$ and that  $H_{2}(X,\mathbb{Z})=\bigoplus
_{\alpha_{i}\in\Delta\setminus\Delta_{P}}  \mathbb{Z}\sigma_{s_{i}}$ can be
canonically identified with  $Q^{\vee}/Q^{\vee}_{P}$, where $Q^{\vee}%
_{P}:=\bigoplus_{\alpha\in\Delta_{P}}\mathbb{Z}\alpha^{\vee}$.  For each
$\alpha_{j}\in\Delta\setminus\Delta_{P}$, we introduce a formal variable
$q_{\alpha_{j}^{\vee}+Q^{\vee}_{P}}$.  For $\lambda_{P}=\sum_{\alpha_{j}%
\in\Delta\setminus\Delta_{P}}a_{j}\alpha_{j}^{\vee}+Q^{\vee}_{P}\in H_{2}(X,
\mathbb{Z})$,  we denote $q_{\lambda_{P}}=\prod_{\alpha_{j}\in\Delta
\setminus\Delta_{P}}q_{\alpha_{j}^{\vee}+Q^{\vee}_{P}}^{a_{j}}$.  The
(\textbf{small}) \textbf{quantum cohomology} $QH^{*}(X)=(H^{*}(X)\otimes
\mathbb{Q}[\mathbf{q}], \star)$ of $X$  is a commutative ring and has a
$\mathbb{Q}[\mathbf{q}]$-basis of Schubert classes $\sigma^{u}=\sigma
^{u}\otimes1$.  The { structure coefficients} $N_{u, v}^{w, \lambda_{P}}$ for
the  quantum product
\[
\sigma^{u}\star\sigma^{v} = \sum_{w\in W^{P}, \lambda_{P}\in Q^{\vee}/Q^{\vee
}_{P}} N_{u,v}^{w, \lambda_{P}}q_{\lambda_{P}}\sigma^{w}%
\]
are three-pointed genus zero Gromov-Witten invariants and they are all
\textbf{non-negative}.

When $P=B$, we have $\Delta_{P}=\emptyset, Q^{\vee}_{P}=0$, $W_{P}=\{1\}$ and
$W^{P}=W$. In this case, we simply denote $\lambda=\lambda_{P}$ and
$q_{j}=q_{\alpha_{j}^{\vee}}$. It is well-known that  $N_{u, v}^{w, \lambda
}=0$ unless both of the followings hold:

\begin{enumerate}

\item $\ell(w)+\langle2\rho, \lambda\rangle=\ell(u)+\ell(v)$ (which comes from
the dimension constraint);

\item $\lambda$ is \textit{effective}, i.e. $\lambda=\sum_{j=1}^{n}
a_{j}\alpha_{j}^{\vee}$ with $a_{j}\in\mathbb{Z}_{\geq0}$ for all $j$.
\end{enumerate}

\subsection{Preliminaries}

\label{subsectreviewmainresults} In this subsection, we collect some known
propositions. As we will see in the next section, we give the quantum Pieri
rules for Grassmannians of classical types, based on the main result of
\cite{leungliQtoC} (see Proposition \ref{propmainthminLeungLi}), the
Peterson-Woodward comparison formula (see Proposition \ref{comparison}) and
the quantum Chevalley formula (see Proposition \ref{quanchevalley}).

As in \cite{leungliQtoC}, given any simple root $\alpha\in\Delta$, we define a
map $\mbox{sgn}_{\alpha}$ as follows.
\[
\mbox{sgn}_{\alpha}: W\rightarrow\{0, 1\};\,\, \mbox{sgn}_{\alpha}(w)=%
\begin{cases}
1, & \mbox{if } \ell(w)-\ell(ws_{\alpha})>0\\
0, & \mbox{if } \ell(w)-\ell(ws_{\alpha})\leq0
\end{cases}
.
\]
\noindent It is well-known that (see e.g. \cite{humr})  $\mbox{sgn}_{\alpha
}(w)=0$ if and only if $w(\alpha)\in R^{+}$.

The following proposition relates numbers of rational curves representing
``different" homology classes of $G/B$.

\begin{prop}
[Theorem 2.2 of \cite{leungliQtoC}]\label{propmainthminLeungLi}  For any $u,
v, w\in W$ and $\lambda\in Q^{\vee}$, we have

\begin{enumerate}

\item $N_{u, v}^{w, \lambda}=0$ unless {\upshape $\mbox{sgn}_{\alpha
}(w)+\langle\alpha, \lambda\rangle\leq\mbox{sgn}_{\alpha}%
(u)+\mbox{sgn}_{\alpha}(v)$} for all $\alpha\in\Delta.$

\item Suppose {\upshape $\mbox{sgn}_{\alpha}(w)+\langle\alpha, \lambda
\rangle=\mbox{sgn}_{\alpha}(u)+\mbox{sgn}_{\alpha}(v)=2$} for some $\alpha
\in\Delta$, then
 {\upshape\begin{align}
   \label{eq1}  N_{u, v}^{w, \lambda}&=N_{us_{\alpha}, vs_{\alpha}}^{w, \lambda-\alpha^{\vee}},  &{ whenever } \mbox{ sgn}_{\alpha}(w)=0 \mbox{ or } 1; \\
   \label{eq2}  N_{u, v}^{w, \lambda}&=N_{u, vs_{\alpha}}^{ws_{\alpha}, \lambda-\alpha^{\vee}}, &{ i\!f } \mbox{ sgn}_{\alpha}(w)=0;\\
   \label{eq3}  N_{u, v}^{w, \lambda}&=N_{u, vs_{\alpha}}^{ws_{\alpha}, \lambda},   &{ i\!f } \mbox{ sgn}_{\alpha}(w)=1.
\end{align}
}
\end{enumerate}
\end{prop}

\noindent  As a consequence, we obtain the next vanishing criterion for
the Gromov-Witten invariants\footnote{By ``Gromov-Witten invariants", we always
mean  ``three-pointed genus zero Gromov-Witten invariants" in this paper.}
$N_{u, v}^{w, \lambda}$.

\begin{cor}
\label{corovanish} For any $u,v,w\in W$ and $\lambda\in Q^{\vee}$, we have
$N_{u,v}^{w,\lambda}=0$ whenever there exists $\alpha\in\Delta$ such that one
of the followings holds.

\begin{enumerate}

\item $\langle\alpha, \lambda\rangle=2$ and $N_{u, vs_{\alpha}}^{ws_{\alpha},
\lambda-\alpha^{\vee}}=0$;

\item $\langle\alpha, \lambda\rangle=1$, {\upshape $\mbox{sgn}_{\alpha}(u)=0$}
and $N_{u, vs_{\alpha}}^{ws_{\alpha}, \lambda-\alpha^{\vee}}=0$;

\item $\langle\alpha,\lambda\rangle=0$, {\upshape$\mbox{sgn}_{\alpha
}(u)=\mbox{sgn}_{\alpha}(v)=0$} and $N_{u,vs_{\alpha}}^{ws_{\alpha},\lambda
}=0$.
\end{enumerate}
\end{cor}

\begin{proof}
Note that $\mbox{sgn}_{\alpha}$ is a map frow $W$ to $\{0, 1\}$. Thus if any
one of the above three assumptions holds, we have  $\mbox{sgn}_{\alpha
}(w)+\langle\alpha, \lambda\rangle\geq\mbox{sgn}_{\alpha}%
(u)+\mbox{sgn}_{\alpha}(v)$.

When the inequality ``$>$" holds, we are already done  by using Proposition
\ref{propmainthminLeungLi} (1).  Now we assume the equality ``$=$" holds.  As
a consequence, if assumption (2) holds,  then we have $\mbox{sgn}_{\alpha
}(w)=0$ and $\mbox{sgn}_{\alpha}(v)=1$. Applying Proposition
\ref{propmainthminLeungLi} (2)  for $u^{\prime}=v, v^{\prime}=us_{\alpha}$ and
$w^{\prime}=ws_{\alpha}$,  we deduce that $N_{u, v}^{w, \lambda}=N_{v, u}^{w,
\lambda}=N_{u^{\prime}, v^{\prime}s_{\alpha}}^{w^{\prime}s_{\alpha}, \lambda
}=N_{u^{\prime}s_{\alpha}, v^{\prime}s_{\alpha}}^{w^{\prime}, \lambda
-\alpha^{\vee}}  =N_{vs_{\alpha}, u}^{ws_{\alpha}, \lambda-\alpha^{\vee}}
=0$.  Similarly, we can  show $N_{u, v}^{w, \lambda}= 0$ whenever either
assumption (1) or assumption (3) holds.
\end{proof}
We will use Proposition \ref{propmainthminLeungLi} and Corollary \ref{corovanish}
   very frequently in section 3. Whenever necessary, we will point out what we are applying explicitly, by using the words ``Applying (\textit{reference}) to $(u', v', w', \lambda', \alpha')$".

The next identity for certain classical intersection numbers is also a direct consequence of Proposition \ref{propmainthminLeungLi}.
\begin{cor}
\label{coridentity} Let $u,v,w\in W$ and  $\alpha\in\Delta$. Suppose  {\upshape  $\mbox{sgn}_{\alpha}(w)=\mbox{sgn}_{\alpha}(u)+1=1$},
 then $N_{u, v}^{w, 0}$ is equal to $N_{u, vs_\alpha}^{ws_\alpha, 0}$ if {\upshape $\mbox{sgn}_{\alpha}(v)=1$}, or $0$ otherwise.
\end{cor}

\begin{proof}
   If  $\mbox{sgn}_{\alpha}(v)=0$, then $N_{u, v}^{w, 0}=0$ follows directly from Proposition \ref{propmainthminLeungLi} (1).
    If  $\mbox{sgn}_{\alpha}(v)=1$, then we have
       $N_{v, us_\alpha}^{ws_\alpha, \alpha^\vee}=N_{vs_\alpha, u}^{ws_\alpha, 0}=N_{v, u}^{w, 0}$, by Proposition \ref{propmainthminLeungLi} (2).
\end{proof}

\bigskip

The next comparison formula tells us that every Gromov-Witten invariant
($N_{u, v}^{w, \lambda_{P}}$) for $G/P$ equals  a certain Gromov-Witten
invariant ($N_{u, v}^{w\omega_{P}\omega_{P^{\prime}}, \lambda_{B}}$) for
$G/B$.

\begin{prop}
[Peterson-Woodward comparison formula \cite{wo}; see also \cite{lamshi}%
]\label{comparison}  ${}$

\begin{enumerate}

\item Let $\lambda_{P}\in Q^{\vee}/Q_{P}^{\vee}$. Then there is a unique
$\lambda_{B}\in Q^{\vee}$ such that $\lambda_{P}=\lambda_{B}+Q_{P}^{\vee}$
and  $\langle\gamma, \lambda_{B}\rangle\in\{0, -1\}$ for all $\gamma\in
R^{+}_{P} \,\,(=R^{+}\cap\bigoplus_{\beta\in\Delta_{P}}\mathbb{Z}\beta)$.

\item Denote $\Delta_{P^{\prime}}:=\{\beta\in\Delta_{P}~|~\langle\beta,
\lambda_{B}\rangle=0\}$. For every $u, v, w\in W^{P}$, we have
\[
N_{u,v}^{w, \lambda_{P} }=N_{u, v}^{w\omega_{P}\omega_{P^{\prime}},
\lambda_{B}}.
\]
Here  $\omega_{P}$ (resp. $\omega_{P'}$) is the longest element
in the Weyl subgroup $W_P$ (resp. $W_{P'}$).
\end{enumerate}
\end{prop}

Thanks to the above comparison formula, we obtain an injection of vector
spaces
\[
\psi_{\Delta, \Delta_{P}}: QH^{*}(G/P)\longrightarrow QH^{*}(G/B)
\,\,\mbox{defined by } q_{\lambda_{P}}\sigma^{w}\mapsto q_{\lambda_{B}}%
\sigma^{w\omega_{P}\omega_{P^{\prime}}}.
\]
We denote by $P_{\alpha}$ the parabolic subgroup (containing $B$)  that
corresponds to the special case of a singleton subset $\{\alpha\}\subset
\Delta$, and simply denote $\psi_{\alpha}=\psi_{\Delta, \{\alpha\}}$. Note
that $R_{P_{\alpha}}^{+}=\{\alpha\}$ and $Q_{P_{\alpha}}^{\vee}=\mathbb{Z}%
\alpha^{\vee}$. In addition, we have the natural fibration  $P_{\alpha
}/B\rightarrow G/B\rightarrow G/P_{\alpha}\mbox{ with } P_{\alpha}%
/B\cong\mathbb{P}^{1}.$

The (Peterson's) quantum Chevalley formula, proved in \cite{fw}, describes the
quantum product of two Schubert classes when one of them is given by a simple reflection.

\begin{prop}
[Quantum Chevalley formula for $G/B$ ]\label{quanchevalley}  For $u\in W$,
$1\leq i\leq n$,
\[
\sigma^{u}\star\sigma^{s_{i}} =\sum\langle\chi_{i}, \gamma^{\vee}\rangle
\sigma^{us_{\gamma}}+\sum\langle\chi_{i}, \gamma^{\vee}\rangle q_{\gamma
^{\vee}}\sigma^{us_{\gamma}},
\]
where the first sum is over positive roots $\gamma$ for which $\ell
(us_{\gamma})=\ell(u)+1$, and the second sum is over positive roots  $\gamma$
for which $\ell(us_{\gamma})=\ell(u)+1-\langle2\rho, \gamma^{\vee}\rangle$.
\end{prop}
\noindent  When $\Delta$ is of $A$-type, the above formula  is also called the quantum Monk formula.

In addition, we will need the next two  lemmas.

\begin{lemma}[see Lemma 3.9 of
\cite{leungli33}]
\label{lengthofpos222}  Let $v\in W$ and $\gamma\in R^{+}$ satisfy
$\ell(vs_{\gamma})=\ell(v)+1-\langle2\rho, \gamma^{\vee}\rangle$.  Then for
any $\alpha_{j}\in\Delta$ with $\langle\alpha_{j}, \gamma^{\vee}\rangle>0$, we
have  $\ell(vs_{\gamma}s_{j})=\ell(vs_{\gamma})+1.$
\end{lemma}

\begin{lemma}\label{lemmasgnV=0}
   For  $v\in W^P$ and   $\alpha\in \Delta_P$, we have {\upshape $\mbox{sgn}_\alpha(v)=0$}.
\end{lemma}
\begin{proof}
   It is a well known fact that $v\in W^P$ and   $\alpha\in \Delta_P$ imply $v(\alpha)\in R^+$. Thus, $\ell(vs_\alpha)>\ell(v)$, and then the statement follows.
\end{proof}
\section{Quantum Pieri rules for Grassmannians of classical types: classical aspects}

In this section, we study the quantum Pieri rules for Grassmannians of
classical types, which describe the quantum product of general Schubert
classes with  the Chern classes of the dual of the tautological subbundles.
We will  only deal with Grassmannians
of type $B, C, D$ here, and  will reprove the
 well-known quantum Pieri rules for Grassmannians of type $A$ (i.e. complex
Grassmannians) \cite{bert} in the appendix.  Furthermore, we will only reduce all the relevant Gromov-Witten invariants in the quantum Pieri rules to certain classical intersection numbers.
As we will see in next section, further reductions can be taken so that these rules can be reformulated in a traditional way.

\subsection{Grassmannians of type $B, C, D$}\label{subsectisotroGrassm} In this subsection, we review some facts on
Grassmanians of type $C,B,D$, i.e. the quotients of Lie groups $G$ of the
aforementioned types by their maximal parabolic subgroups $P$. More details on
these facts can be found for example in \cite{tamvakis} and \cite{BKT2}. We
also illustrate the idea of our proof of quantum Pieri rules.

Every such  Grassmannian $X$ parameterizes isotropic subspaces in a vector
space $E=\mathbb{C}^{N}$ equipped with a standard non-degenerated bilinear
form $(\cdot, \cdot )$ which is  skew-symmetric in the $C$
case and symmetric in the $B$ or $D$ cases. Thus it is usually called an \textit{isotropic Grassamannian} and it
can be described explicitly as follows. The maximal parabolic subgroup $P$
corresponds to a subset $\Delta_{P}=\Delta\setminus\{\alpha_{k}\}$ (we use the
convention of labelling the base as in Humphreys' book \cite{hum}) and the
space $X$ is given by

\begin{enumerate}

\item[(i)] $IG(k, 2n) =\{V\leqslant\mathbb{C}^{2n}~|~ \dim_{\mathbb{C}}V=k,\,
 (V, V)=0\} \mbox{ for type } C_{n}$;

\item[(ii)] $OG(k, 2n+1) =\{V\leqslant\mathbb{C}^{2n+1}~|~ \dim_{\mathbb%
{C}}V=k,\, (V, V)=0\} \mbox{ for type } B_{n}$;

\item[(iii)] $OG(k, 2n+2) =\{V\leqslant\mathbb{C}^{2n+2}~|~ \dim_{\mathbb%
{C}}V=k,\, (V, V)=0\}$, if $G$ is of type $D_{n+1}$ and $1\leq k\leq n-1$;

\item[(iv)] a connected component $OG^{o}(n+1, 2n+2)$ of $OG(n+1, 2n+2)$, if
$G$ is of type $D_{n+1}$ and $k\in\{n, n+1\}$.
\end{enumerate}

In the first three cases, we have $N=2n$, $2n+1$ and $2n+2$ respectively. For
convenience, we have assumed $G$ to be of type $D_{n+1}$, rather than $D_{n}$,
in case (iii) and (iv).  Furthermore when this holds, we can always assume
$k\leq n-1$, since case (iv) can be reduced to case (ii) (see Remark
\ref{remarkBCD}). Customarily, $IG(n, 2n)$ (resp. $OG(n, 2n+1)$, $OG^{o}(n+1,
2n+2)$) is called a \textit{Lagrangian} (resp. \textit{odd orthogonal},
\textit{even orthogonal}) \textit{Grassmannian}.

There are tautological bundles over the isotropic Grassmannian $G/P$:
\[
0\rightarrow\mathcal{S}\rightarrow E\rightarrow\mathcal{Q}\rightarrow0.
\]
The Chern classes of the dual of the tautological subbundle $\mathcal{S}$ are
given by the Schubert classes  $c_{p}(\mathcal{S}^{*})=\sigma^{u}$   (where
$1\leq p\leq k$) with
\[
u:={s_{k-p+1}\cdots s_{k-1}s_{k}},
\]
possibly up to a scale factor of $2$.
Precisely, for case (i), (ii) and (iii), we always have
 $c_{p}(\mathcal{S}^{*})=\sigma^{u}$, except for the special case of $k=n$ for case (ii). Furthermore  for this exceptional case, we  have
 $c_{p}(\mathcal{S}^{*})=2\sigma^{u}$  (see e.g. \cite{pragrat}, \cite{pragratEvenorth}).   Note that case (iv) has been reduced to the exceptional case.
  In the next two subsections, we will show the classical aspects of  the quantum Pieri
rules with respect to $c_{p}(\mathcal{S}^{*})$ (or equivalently $c_{p}(\mathcal{S})=(-1)^pc_{p}(\mathcal{S}^*)$). That is, we give a formula for
the quantum product $\sigma^{u}\star\sigma^{v}$ of a general Schubert class
$\sigma^{v}$ in $QH^{*}(G/P)$ with such a special Schubert class  $\sigma^{u}$, in which we reduce all the Gromov-Witten invariants to classical intersection numbers.

Note that the quantum Pieri rules with respect to $c_{p}(\mathcal{Q})$ have
been given by Buch, Kresch and Tamvakis \cite{BKT2}. In contrast to complex
Grassmannians, (quantum) Pieri rules with respect to $c_{p}(\mathcal{Q})$ do
not imply  (quantum) Pieri rules with respect to $c_{p}(\mathcal{S}^{*})$,
whenever $1< k<n$ (see the next remark and note that $c_1(\mathcal{S}^*)=c_1(\mathcal{Q}$)).

\begin{remark}
[See e.g. \cite{tamvakis} and \cite{BKT2}]\label{remarkBCD}${}$

\begin{enumerate}

\item $OG(n+1, 2n+2)$ has two isomorphic connected components, either of which
is projectively isomorphic to $OG(n, 2n+1)$.

\item $OG(n, 2n+2)$ is a flag variety $G/\bar P$ of $D_{n+1}$-type with
$\Delta_{\bar P}=\Delta\setminus\{\alpha_{n}, \alpha_{n+1}\}$.

\item For any Lagrangian or orthogonal Grassmannian (i.e. for $k=n$), we have
$c_{p}(\mathcal{S}^{*})=c_{p}(\mathcal{Q})$ whenever {\upshape $p\leq
\mbox{rank}(\mathcal{S})$}.
\end{enumerate}
\end{remark}

The idea of our proof of such a formula is as follows. For the isotropic
Grassmannian $G/P$, we have $H_{2}(G/P, \mathbb{Z})=Q^{\vee}/Q_{P}^{\vee}%
\cong\mathbb{Z}$ and consequently  $QH^{*}(G/P)$ contains only one quantum
variable $t:=q_{\alpha_{k}^{\vee}+Q_{P}^{\vee}}$. Thus we can write
\[
\sigma^{u}\star\sigma^{v}=\sigma^{u}\cup\sigma^{v}+\sum_{w\in W^{P}, d\geq
1}N_{u, v}^{w, d}\sigma^{w} t^{d}.
\]
Here we have $N_{u, v}^{w, d}=N_{u, v}^{w, \lambda_{P}}$ where $\lambda
_{P}:=d\alpha_{k}^{\vee}+Q_{P}^{\vee}$, compared with the previous notations.
In addition, we have $N_{u, v}^{w, \lambda_{P}}=N_{u, v}^{\tilde w,
\lambda_{B}}$, where $\tilde w=w \omega_{P}\omega_{P^{\prime}}$ and
$\lambda_{B}$ are  given by the Peterson-Woodward comparison formula.
We can show

\begin{enumerate}

\item $N_{u, v}^{w, d}=0$ unless $d$ is small enough (for instance $d\leq2$).

\item For a small $d$, $N_{u, v}^{\tilde w, \lambda_{B}}$ is equal to a
certain classical intersection number $N_{u^{\prime}, v^{\prime}}^{w^{\prime},
0}$ for  which the classical Pieri rules (or other known formulas) can be
applied.
\end{enumerate}

The dimension constraint may also be helpful. That is, we have $N_{u, v}^{w,
d}=0$ unless $\ell(u)+\ell(v)=\ell(w)+d\cdot\deg t$. Here  we have $\deg t=
2n+1-k$ (resp. $2n-k$, $2n+1-k$) for case (i) (resp. (ii), (iii)) if $k<n$, and
$\deg t=n+1$ (resp. $2n$) for case (i) (resp. (ii) or (iv)) if $k=n$.

In fact, the above method can also been used to recover the well-known quantum
Pieri rules for complex Grassmannians. Details will be given in the appendix.

Note that  whenever referring to $N_{u, v}^{\tilde w, \lambda_{B}}$ (resp.
$N_{u, v}^{w, \lambda_{P}}$ or $N_{u, v}^{w, d}$), we are discussing the
quantum product  $\sigma^{u}\star_{B}\sigma^{v}$ in $QH^{*}(G/B)$ (resp.
$\sigma^{u}\star_{P}\sigma^{v}$ in $QH^{*}(G/P)$).

Due to the above assumptions on $\Delta$, the Dynkin diagram of  $\{\alpha_1, \cdots, \alpha_{n-1}\}$ is of type $A_{n-1}$ in the standard way.
As a consequence, we have the following fact on certain products in the Weyl subgroup generated by $\{s_1, \cdots, s_{n-1}\}$.

\begin{lemma}[Lemma 3.3 of \cite{leungli33}]\label{lemmaprodAtype}
  For   $1\leq i\leq j\leq m<n$ and $1\leq r\leq m$, we have
      $$(s_is_{i+1}\cdots s_j)(s_{r}s_{r+1}\cdots s_m)= \left\{\begin{array}{ll} (s_rs_{r+1}\cdots s_m)(s_is_{i+1}\cdots s_j)  ,&\mbox{if } r\geq j+2\\
                                 s_is_{i+1}\cdots s_m, &\mbox{if } r=j+1\\
                                  (s_{r+1}s_{r+2}\cdots s_m)(s_is_{i+1}\cdots s_{j-1}),&\mbox{if  }  i\leq r\leq j    \\
                              (s_rs_{r+1}\cdots s_m)(s_{i-1}s_{i}\cdots s_{j-1}),&\mbox{if }
                                  r<i\end{array}\right.\!\!.
            $$
\end{lemma}

\subsection{Classical aspects of quantum Pieri rules for Grassmannians of type $C$}

Throughout this subsection, we consider a Grassmannian of type $C_{n}$.
Precisely, we consider  the isotropic Grassmannian $G/P=IG(k, 2n)$.  Thus the
base $\Delta$ is of type $C_{n}$. Unless otherwise stated, we will always use the following definition of $u$ in the rest of this section.
\begin{defn} Fix $1\leq p \leq k$, we define
   $$u:=s_{k-p+1}\cdots s_{k-1}s_k.$$
\end{defn}

To show Theorem \ref{thmQPRforIGIG},  the main result of this subsection,
we need to compute all the Gromov-Witten invariants $N_{u, v}^{w, d}$ for the
quantum product $\sigma^{u}\star\sigma^{v}$.  Recall that for a given $d$, we
have  $N_{u, v}^{w, d}=N_{u, v}^{w, \lambda_{P}}=N_{u, v}^{\tilde w,
\lambda_{B}}$, where $\lambda_{P}=d\alpha_{k}^{\vee}+Q_{P}^{\vee}$,  and
$\tilde w=w\omega_{P}\omega_{P^{\prime}}$, $\lambda_{B}$ are both defined in
Proposition \ref{comparison}. For each $j$, we simply denote $\mbox{sgn}_{j}%
:=\mbox{sgn}_{\alpha_{j}}$. 

\begin{lemma}
\label{lemmalambdaBfortypeCC}  Write $d=mk+r$ where $m, r\in\mathbb{Z}$ with
$1\leq r\leq k$.  Then we have $\lambda_{B}=\lambda^{\prime}$ with
$\lambda^{\prime}:=m\sum_{j=1}^{k-1}j\alpha_{j}^{\vee}+\sum_{j=1}^{r-1}%
j\alpha_{k-r+j}^{\vee}+ d\sum_{j=k}^{n}\alpha_{j}^{\vee}$.
\end{lemma}

\begin{proof}
It is easy to check that  $\langle\alpha_{i}, \lambda^{\prime}\rangle=-1$ if
$i=k-r$, or $0$ otherwise. Hence, $\langle\gamma, \lambda^{\prime}\rangle
\in\{0, -1\}$ for all $\gamma\in R_{P}^{+}$.  Thus the statement follows from
the uniqueness of $\lambda_{B}$ (see Proposition \ref{comparison}).
\end{proof}
\begin{lemma}\label{lemmaAlphakLambdaB}
  With the same notations as in     Lemma \ref{lemmalambdaBfortypeCC}, we have
     $\langle\alpha_k, \lambda_B\rangle=m+1$ if $k<n$, or  $2(m+1)$ if $k=n$.
\end{lemma}
\begin{proof}
   If $k=n$,  we have $\langle\alpha_{k}, \lambda_{B}\rangle
   =\langle \alpha_n, m(n-1)\alpha_{n-1}^\vee+(r-1)\alpha_{n-1}^\vee+(mn+r)\alpha_n^\vee\rangle =2(m+1)$, by noting   $\langle \alpha_n, \alpha_j^\vee\rangle =0$ for all $j<n-1$.
 If $k<n$, we have  $\langle\alpha_{k}, \lambda_{B}\rangle
   =\langle \alpha_k, m(k-1)\alpha_{k-1}^\vee+(r-1)\alpha_{k-1}^\vee+(mk+r)\alpha_k^\vee+(mk+r)\alpha_{k+1}^\vee\rangle =m+1$.
\end{proof}
\begin{lemma}
\label{lemmavanishforCCC}  Suppose $d\geq k+1$, then we have $N_{u, v}^{w,
d}=0$ for any $w\in W^{P}$.
\end{lemma}

\begin{proof}
Since $d\geq k+1$, we have $d=mk+r$ with $1\leq r\leq k$ and $m\geq 1$.

When $k=n$, we have $\langle\alpha_{k}, \lambda_{B}\rangle
   =2(m+1)>2$ by Lemma \ref{lemmaAlphakLambdaB}.
Thus we have   $N_{u, v}^{w, d}=N_{u, v}^{\tilde w, \lambda_{B}}=0$ by
Proposition \ref{propmainthminLeungLi} (1).

When $k<n$,  we have  $\langle\alpha_{k}, \lambda_{B}\rangle=m+1\geq2$, by Lemma \ref{lemmaAlphakLambdaB} again.
 If ``$>$" holds, then we are already done  by using Proposition \ref{propmainthminLeungLi} (1) again.
 Note $v\in W^P$ and $\alpha_{k+1}\in \Delta_P$. By Lemma \ref{lemmasgnV=0}, we have $\mbox{sgn}_{k+1}(v)=0$.  If
$\langle\alpha_{k}, \lambda_{B}\rangle=2$,  then we have  $\mbox{sgn}_{{k+1}}%
(us_{k})+\mbox{sgn}_{{k+1}}(v)=0<1\leq\mbox{sgn}_{{k+1}}(\tilde w
s_{k})+\langle\alpha_{k+1}, \lambda_{B}-\alpha_{k}^{\vee}\rangle$. Thus we have
 $N_{us_{k},
v}^{\tilde w s_{k}, \lambda_{B}-\alpha_{k}^{\vee}}=0$
by Proposition \ref{propmainthminLeungLi}(1). Consequently,  we   still have
 $N_{u, v}^{w, d}=N_{u, v}^{\tilde w, \lambda_{B}}=0$  by
Corollary \ref{corovanish} (1).
\end{proof}

\begin{remark}
  The above lemma also follows directly from the dimension count.
\end{remark}

\begin{lemma}
\label{lemvanforCCCkknn}  Let $u^{\prime}, v^{\prime}, w^{\prime}\in W$ and
$\lambda\in Q^{\vee}$. For the quantum product  $\sigma^{u^{\prime}}%
\star\sigma^{v^{\prime}}$ in $QH^{*}(G/B)$, the structure constant
$N_{u^{\prime}, v^{\prime}}^{w^{\prime}, \lambda}$ vanishes, if both  {\upshape (a)}
  and  {\upshape (b)} hold:  {\upshape
\[
\mbox{(a)}\,\, \langle\alpha_{n}, \lambda\rangle=2, \langle\alpha_{n-1},
\lambda\rangle=0; \quad\mbox{(b)}\,\,  \mbox{sgn}_{n}(u^{\prime}s_{n}%
s_{n-1})=0, \mbox{sgn}_{n-1}(v^{\prime})=0.
\]
}
\end{lemma}

\begin{proof}
Note that $\mbox{sgn}_{n}(u^{\prime}s_{n}s_{n-1})+\mbox{sgn}_{n}(v^{\prime
})=0+\mbox{sgn}_{n}(v^{\prime})\leq1<2=\langle\alpha_{n}, \lambda\rangle=
\langle\alpha_{n}, \lambda-\alpha_{n}^{\vee}-\alpha_{n-1}^{\vee}\rangle$.  By
Proposition \ref{propmainthminLeungLi} (1), we have  $N_{u^{\prime}%
s_{n}s_{n-1}, v^{\prime}}^{w^{\prime}s_{n}s_{n-1}, \lambda-\alpha_{n}^{\vee
}-\alpha_{n-1}^{\vee}}=0$.  Since $\mbox{sgn}_{n-1}(v^{\prime})=0$ and
$\langle\alpha_{n-1}, \lambda-\alpha_{n}^{\vee}\rangle=-\langle\alpha_{n-1},
\alpha_{n}^{\vee}\rangle=1$,  we have  $N_{u^{\prime}s_{n}, v^{\prime}%
}^{w^{\prime}s_{n}, \lambda-\alpha_{n}^{\vee}}=0$ by Corollary
\ref{corovanish} (2).  Consequently, we have  $N_{u^{\prime}, v^{\prime}%
}^{w^{\prime}, \lambda}=0$ by Corollary \ref{corovanish} (1).
\end{proof}

The next proposition shows us that $t$ is the largest  power $t^{d}$ appearing
in the quantum  product $\sigma^{u}\star\sigma^{v}$ in $QH^{*}(G/P)$.

\begin{prop}
\label{propdlarger2CCC}  Suppose $d\geq2$, then we have $N_{u, v}^{w, d}=0$
for any $w\in W^{P}$.
\end{prop}

\begin{proof}
We can assume $2\leq d\leq k$ due to Lemma \ref{lemmavanishforCCC}.  It
suffices to show $N_{u, v}^{\tilde w, \lambda_{B}}=0$, where we note
$\lambda_{B}= \sum_{j=1}^{d-1}j\alpha_{k-d+j}^{\vee}+d\sum_{j=k}^{n}\alpha
_{j}^{\vee}$.  Consequently, $\langle\alpha_{k-1}, \lambda_{B} \rangle=0$.

Suppose $k=n$, then $\langle\alpha_{n}, \lambda_{B} \rangle=2$.
Clearly, we have $\mbox{sgn}_{n}(us_{n}s_{n-1})=0$  and $\mbox{sgn}_{n-1}%
(v)=0$. Thus we are done by Lemma \ref{lemvanforCCCkknn}.

Now we assume $k<n$.
Since $\mbox{sgn}_{k+1}(u)=\mbox{sgn}_{k+1}(v)=0$ and $\langle\alpha_{k+1},
\lambda_{B}\rangle=0$,  it suffices to show $N_{us_{k+1}, v}^{\tilde w
s_{k+1}, \lambda_{B}}=0$, due to Corollary \ref{corovanish} (3).  Note
$\mbox{sgn}_{k}(us_{k+1})=0$ and $\langle\alpha_{k}, \lambda_{B}\rangle=1$.
Then it suffices to show  $N_{us_{k+1}, vs_{k}}^{\tilde w s_{k+1}s_{k},
\lambda_{B}-\alpha_{k}^{\vee}}=0$, due to Corollary \ref{corovanish} (2).  For
any $1\leq i\leq m\leq n$, we denote  $v_{i}^{(m)}:= s_{m}s_{m-1}\cdots
s_{m-i+1}.$ By induction on $i$, we reduce the above statement to the
following one.  To show $N_{us_{k+1}, vs_{k}}^{\tilde w s_{k+1}s_{k},
\lambda_{B}-\alpha_{k}^{\vee}}=N_{us_{k+1}, vv_{1}^{(k)}}^{\tilde w
s_{k+1}v_{1}^{(k)}, \lambda_{B}-\sum_{j=1}^{1}\alpha_{k-j+1}^{\vee}}=0$, it
suffices to show $N_{us_{k+1}, vv^{(k)}_{d}}^{\tilde w s_{k+1}v^{(k)}_{d},
\lambda_{B}-\sum_{j=1}^{d}\alpha_{k-j+1}^{\vee}}=0$. Furthermore by induction
on $m$, it suffices to show  $N_{u^{\prime}, v^{\prime}}^{w^{\prime},
\lambda_{B}^{\prime}}=0$, in which $u^{\prime}=us_{k+1}\cdots s_{n-1}s_{n}$,
$v^{\prime}=vv^{(k)}_{d}\cdots v^{(n-1)}_{d}$, $w^{\prime}=\tilde
ws_{k+1}\cdots s_{n-1}s_{n}v^{(k)}_{d}\cdots v^{(n-1)}_{d}$ and $\lambda
_{B}^{\prime}=\lambda_{B}-\sum_{j=1}^{d}\alpha_{k-j+1}^{\vee}-\cdots
-\sum_{j=1}^{d}\alpha_{n-1-j+1}^{\vee}=\sum_{j=1}^{d}j\alpha_{n-d+j}^{\vee}$.
(Here we always use Corollary \ref{corovanish} (2), (3) for the inductions.)

Note that $\langle\alpha_{n}, \lambda_{B}^{\prime}\rangle=2, \langle
\alpha_{n-1}, \lambda_{B}^{\prime}\rangle=0$ and $\mbox{sgn}_{n}(u^{\prime
}s_{n}s_{n-1})=0$. In addition, we note that $d\geq2$, so that $v_{d}%
^{(j)}(\alpha_{j})=\alpha_{j-1}$ for each $k\leq j\leq n-1$.  Thus we have
$v^{\prime}(\alpha_{n-1})=v(\alpha_{k-1})\in R^{+}$ and consequently
$\mbox{sgn}_{n-1}(v^{\prime})=0$. Hence, we do have $N_{u^{\prime}, v^{\prime}}^{w^{\prime},
\lambda_{B}^{\prime}}=0$, by using Lemma
\ref{lemvanforCCCkknn}.
\end{proof}

\begin{remark}
Proposition \ref{propdlarger2CCC} (resp. Proposition \ref{propdlarger2BBDD})
can also be proved by using Theorem 1.3 (d)  (resp. Theorem 2.3 (d) and
Theorem 3.3 (d)) of \cite{BKT2}.
\end{remark}

The next lemma also works with exactly the same arguments, for either of the
cases: (1) $\Delta$ is of type $B_{n}$; (2) $\Delta$ is of type $D_{n+1}$ and
$k<n$.

\begin{lemma}
\label{lemforBBBCCCDD}
Let
$u^{\prime}=s_{j-i+1}\cdots s_{j-1}s_{j}$ where  $1\leq i\leq j<k$, and
   $\lambda$ be effective with    $\langle
\chi_{j}, \lambda\rangle=0$. If $\lambda\neq 0$, then we have $N_{u^{\prime},
v^{\prime}}^{w^{\prime}, \lambda}= 0$ for any $v', w'\in W$.
\end{lemma}

\begin{proof}
Note that the product $\big(\sigma^{s_{j}}\big)^{i}:=\sigma^{s_{j}}\star
\cdots\star\sigma^{s_{j}}$ of $i$ copies of $\sigma^{s_{j}}$  is the summation
of $\sigma^{u^{\prime}}$ and other nonnegative  terms. Hence,  $\big(\sigma
^{s_{j}}\big)^{i}\star\, \sigma^{v^{\prime}}=\sigma^{u^{\prime}}\star\,
\sigma^{v^{\prime}}+(\mbox{other \textit{nonnegative} terms})$ $=  N_{u^{\prime},
v^{\prime}}^{w^{\prime}, \lambda}q_{\lambda%
}\sigma^{w^{\prime}}+ (\mbox{other \textit{nonnegative} terms})$.  On the other
hand, we have
\begin{align*}
\big(\sigma^{s_{j}}\big)^{i}\star\, \sigma^{v^{\prime}} & =\sum\limits_{\gamma
_{i}}\cdots\sum\limits_{\gamma_{1}} N_{s_{j}, v^{\prime}}^{v^{\prime}%
s_{\gamma_{1}}, \mu_{1}} N_{s_{j}, v^{\prime}s_{\gamma_{1}}}^{v^{\prime
}s_{\gamma_{1}}s_{\gamma_{2}}, \mu_{2}}\cdots N_{s_{j}, v^{\prime}%
s_{\gamma_{1}}\cdots s_{\gamma_{i-1}}}^{v^{\prime}s_{\gamma_{1}}\cdots
s_{\gamma_{i}}, \mu_{i}} q_{\mu_{1}+\cdots+\mu_{i}}\sigma^{v^{\prime}%
s_{\gamma_{1}}\cdots s_{\gamma_{i}}}\\
&  =\sum\limits_{\gamma_{i}}\cdots\sum\limits_{\gamma_{1}} \prod_{h=1}^{i}
\langle\chi_{j}, \gamma_{h}^{\vee}\rangle q_{\mu_{1}+\cdots+\mu_{i}}%
\sigma^{v^{\prime}s_{\gamma_{1}}\cdots s_{\gamma_{i}}},
\end{align*}
by the quantum Chevalley formula. Here for $1\leq h\leq i$, $\gamma_{h}\in
R^{+}$, $\mu_{h}\in\{0, \gamma_{h}^{\vee}\}$, and they satisfy  $\ell
(v^{\prime}s_{\gamma_{1}}\cdots s_{\gamma_{h}})=\ell(v^{\prime}s_{\gamma_{1}%
}\cdots s_{\gamma_{h-1}})+1 -\langle2\rho, \mu_{h}\rangle$.  If
$N_{u^{\prime}, v^{\prime}}^{w^{\prime},
\lambda}\neq 0$ for some $v', w'$, then there exists a sequence $(\gamma_{1}, \ldots,
\gamma_{i})$ such that  $\lambda=\sum_{h=1}^{i}\mu_{h}$ and
$\prod_{h=1}^{i} \langle\chi_{j}, \gamma_{h}^{\vee}\rangle\neq0$. Since $\lambda\neq0$,  there
exists $1\leq h^{\prime}\leq i$ such that $\mu_{h^{\prime}}\neq0$ and
$0=\langle\chi_{j}, \lambda\rangle= \langle\chi_{j}, \sum
_{h=1}^{i}\mu_{h}\rangle\geq\langle\chi_{j}, \mu_{h^{\prime}}\rangle
=\langle\chi_{j}, \gamma_{h^{\prime}}^{\vee}\rangle>0$. Contradiction.
\end{proof}

The next well-known fact, characterizing $\omega_{P}\omega_{P^{\prime}}$,
works for $\Delta$ of \textit{any type}.

\begin{lemma}
[See e.g. Lemma 3.5 of \cite{leungli33}]\label{charalongest}  An element $\bar
w\in W_{P}$ is equal to $\omega_{P}\omega_{P^{\prime}}$, if both of the
followings hold:  {\upshape$\mbox{(i) } \ell(\bar w)=\ell(\omega_{P}%
\omega_{P^{\prime}})$};  {\upshape $\mbox{(ii) } \bar w(\alpha)\in R^{+}$} for
all $\alpha\in\Delta_{P^{\prime}}$.
\end{lemma}

In the rest of this subsection, we fix   the positive root  $\gamma:=\alpha_n+2\sum_{j=k}^{n-1}\alpha_j$. Note that  $\gamma^{\vee}=\alpha
_{k}^{\vee}+\alpha_{k+1}^{\vee}+\cdots+ \alpha_{n}^{\vee}$ and $\langle 2\rho, \gamma^\vee\rangle=2n-2k+2$.

\begin{thm}
[Classical aspects of quantum Pieri rules for Grassmannians of type $C_{n}$]\label{thmQPRforIGIG}%
${}$
Let $\sigma^{u}, \sigma^{v}$ be Schubert classes in the quantum cohomolgy of
the isotropic Grassmannian $G/P=IG(k, 2n)$.  Recall  $u=s_{k-p+1}\cdots
s_{k-1}s_{k}$, where $1\leq p\leq k$. 
 We have
\[
\sigma^{u}\star\sigma^{v}=\sigma^{u}\cup\sigma^{v} +%
\begin{cases}
t\sum_{w\in W^{P}} N_{us_{k}, vs_{\gamma}}^{ws_{1}\cdots s_{k-1}, 0}\sigma
^{w}, & \mbox{if }   \ell(vs_{\gamma})=\ell(v)-2n+2k-1\\
0, & \mbox{otherwise}
\end{cases}
.
\]

\end{thm}

\begin{proof}
Due to Proposition \ref{propdlarger2CCC}, we have $\sigma^{u}\star\sigma
^{v}=\sigma^{u}\cup\sigma^{v} + t\sum_{w\in W^{P}} N_{u, v}^{w, 1}\sigma^{w}$.
Thus by the Peterson-Woodward comparison formula, it suffices to compute the
Gromov-Witten invariants $N_{u, v}^{w, 1}=N_{u, v}^{w\omega_{P}\omega
_{P^{\prime}}, \lambda_{B}}$ with respect to $\lambda_{P}=\alpha_{k}^{\vee
}+Q^{\vee}_{P}$.  By Lemma \ref{lemmalambdaBfortypeCC}, we have $\lambda
_{B}=\sum_{j=k}^{n}\alpha_{j}^{\vee}=\gamma^{\vee}$, so that $\Delta
_{P^{\prime}}=\Delta_{P}\setminus\{\alpha_{k-1}\}$. Consequently, we have
$\ell(\omega_{P}\omega_{P^{\prime}})=|R_{P}^{+}|-|R_{P^{\prime}}^{+}|=k-1$.
Hence, we conclude $\omega_{P}\omega_{P^{\prime}}=s_{1}s_{2}\cdots s_{k-1}$ by
(easily) checking the assumptions in Lemma \ref{charalongest}. Therefore,  it
is sufficient to compute $N_{u, v}^{ws_{1}\cdots s_{k-1}, \gamma^{\vee}%
}q_{\gamma^{\vee}}\sigma^{ws_1\cdots s_{k-1}}$ in the product $\sigma^{u}\star_{B} \sigma^{v}$
in $QH^{*}(G/B)$.
By abuse of notations, we simply denote ``$\star_{B}$" as ``$\star$" here. We claim

\begin{enumerate}

\item the contribution $N_{u, v}^{ws_{1}\cdots s_{k-1}, \gamma^{\vee}}$ for
$q_{\gamma^{\vee}}\sigma^{ws_{1}\cdots s_{k-1}}$ from $\sigma^{u}\star
\sigma^{v}$ is the same as the contribution $N_{us_{k}, s_{k}, v}%
^{ws_{1}\cdots s_{k-1}, \gamma^{\vee}}$ for $q_{\gamma^{\vee}}\sigma
^{ws_{1}\cdots s_{k-1}}$ from $\sigma^{us_{k}}\star\sigma^{s_{k}}\star
\sigma^{v}$;

\item $N_{us_{k}, s_{k}, v}^{w^{\prime}, \gamma^{\vee}}=N_{s_{k},
v}^{vs_{\gamma}, \gamma^{\vee}} \cdot N_{us_{k}, vs_{\gamma}%
}^{w^{\prime}, 0}$, for any $w^{\prime}\in W$.
\end{enumerate}

\noindent Assuming these claims, if $\ell(vs_{\gamma})\neq\ell(v)+1-\langle2\rho, \gamma^{\vee}\rangle$, then
$N_{s_{k}, v}^{vs_{\gamma}, \gamma^{\vee}} =0$. As a consequence,  $N_{u,
v}^{w, 1}=N_{u, v}^{ws_{1}\cdots s_{k-1}, \gamma^{\vee}}=0\cdot N_{us_{k},
vs_{\gamma}^{\vee}}^{ws_{1}\cdots s_{k-1}, 0}=0$ for any $w\in W^{P}$. Hence,
$\sigma^{u}\star_{P}\sigma^{v}=\sigma^{u}\cup\sigma^{v}$.  If $\ell
(vs_{\gamma})= \ell(v)+1-\langle2\rho, \gamma^{\vee}\rangle$, then we have
$N_{s_{k}, v}^{vs_{\gamma}, \gamma^{\vee}}=\langle\chi_{k}, \gamma^{\vee
}\rangle=1$, by the quantum Chevalley formula. Thus  $N_{u, v}^{ws_{1}\cdots
s_{k-1}, \gamma^{\vee}}=N_{us_{k}, vs_{\gamma}}^{ws_{1}\cdots s_{k-1}, 0}$.
In addition, we note $\ell(v)+1-\langle2\rho, \gamma^{\vee
}\rangle=\ell(v)-2n+2k-1$.
Hence, we are done.

It remains to show claims (1) and (2). If $\ell(u)=1$, then $us_{k}=1$ and we
are done. Thus we assume $\ell(u)>1$ in the rest, and show claim (1) first.
Note that $u=s_{k-p+1}\cdots s_{k-1}s_{k}$ is of length $p$. By the quantum
Chevalley formula, we have $\sigma^{us_{k}}\star\sigma^{s_{k}}=\sigma
^{u}+\sigma^{s_{k}us_{k}}$.  It suffices to show  $\sigma^{s_{k}us_{k}}%
\star\sigma^{v}$ makes no contribution for $q_{\gamma^{\vee}}\sigma
^{ws_{1}\cdots s_{k-1}}$. Indeed, we have $\mbox{sgn}_{j}(s_{k}us_{k})=0$
whenever $j\geq k$, by noting  $s_{k}us_{k}(\alpha_{j})=s_{k}s_{k-p+1}\cdots
s_{k-2}s_{k-1}(\alpha_{j})\in R^{+}$.  Since $\langle\alpha_{k}, \gamma^{\vee
}\rangle=1$, $N_{s_{k}us_{k}, v}^{ws_{1}\cdots s_{k-1}, \gamma^{\vee}}=0$
follows if  $N_{s_{k}us_{k}, vs_{k}}^{ws_{1}\cdots s_{k-1}s_{k}, \gamma^{\vee
}-\alpha_{k}^{\vee}}=0$, by Corollary \ref{corovanish} (2). Repeating this
reduction, it suffices to show $N_{s_{k}us_{k}, vs_{k}\cdots s_{n-1}}%
^{ws_{1}\cdots s_{n-1}, \alpha_{n}^{\vee}}=0$, which does follow by using
Proposition \ref{propmainthminLeungLi} (1) with respect to $\mbox{sgn}_{n}$.
Thus claim (1) follows.

The contribution $N_{us_{k}, s_{k}, v}^{w^{\prime}, \gamma^{\vee}}$ for
$q_{\gamma^{\vee}}\sigma^{w^{\prime}}$ from   $\sigma^{us_{k}}\star\sigma
^{s_{k}}\star\sigma^{v}=(\sigma^{s_{k}}\star\sigma^{v})\star \sigma^{us_{k}}$ is given by $N_{us_{k}, s_{k}, v}^{w^{\prime}, \gamma^{\vee}}=\sum_{w^{\prime\prime}\in W,
\lambda\in Q^{\vee}} N_{s_{k}, v}^{w^{\prime\prime}, \lambda}N_{w^{\prime
\prime}, us_{k}}^{w^{\prime}, \gamma^{\vee}-\lambda}$ (which contains only
finitely many nonzero terms).  Hence, claim (2) becomes a direct consequence
of the quantum Chevalley formula and Lemma \ref{lemforBBBCCCDD}.
\end{proof}

\begin{remark}
Using Proposition \ref{propmainthminLeungLi}, we can also show $N_{u,
v}^{ws_{1}\cdots s_{k-1}, \gamma^{\vee}}=N_{u^{\prime}, v'%
}^{w^{\prime}, 0}$ where $u^{\prime}=s_{k-p+1}\cdots s_{n}$, $v'=vs_{k}\cdots s_{n}$ and $w^{\prime
}=ws_{1}\cdots s_{k-1}s_{k+1}\cdots s_{n} s_{k}\cdots s_{n}$. As a consequence, we
 can apply   a
generalized classical Pieri rule given by Bergeron and Sottile
\cite{bergSottile} to express $N_{u^{\prime}, v'%
}^{w^{\prime}, 0}$ more explicitly.
\end{remark}

\begin{example}
 For $X=IG(2, 8)$, we take $u=s_1s_2, v=s_3s_4s_3s_1s_2$ and $w=\mbox{id}$. Then  $vs_\gamma=vs_2s_3s_4s_3s_2
   =s_1s_2$, so that $\ell(v)=2\neq 0=\ell(v)-2\cdot 4+2\cdot 2-1$. 
   Thus $N_{u, v}^{w, 1}=0$.
   (In terms of notations in Example 1.3 of \cite{BKT2}, we have $\sigma^u=\sigma_{1, 1}, \sigma^v=\sigma_{4, 1}$ and
     $N_{u, v}^{w, 1}=\langle \sigma_{1, 1}, \sigma_{4, 1}, \sigma_{6, 5}\rangle_{1}$.)
\end{example}


  Denote by $\tilde
P\supset B$ the  parabolic subgroup corresponding to the subset $\Delta
\setminus\{\alpha_{k-1}\}$. That is, $G/\tilde P=IG(k-1, 2n)$.   (When $k=1$, we mean $\tilde
P=G$, i.e. $W^{\tilde P}=\{\mbox{id}\}$.) 
Recall that $\gamma^\vee=\sum_{j=k}^n\alpha_j^\vee$ so that  $\ell(v)+1-\langle2\rho, \gamma^{\vee
}\rangle=\ell(v)-2n+2k-1$.

\begin{lemma}
\label{lemTFAE}  For any  $v\in W^{P}$, the following are equivalent:  {\upshape
\[
\mbox{(a) } \ell(vs_{\gamma})=\ell(v)+1-\langle2\rho, \gamma^{\vee}\rangle;
\qquad\mbox{(b) } vs_{\gamma}(\alpha_{k})\in R^{+};\qquad\mbox{(c) }
vs_{\gamma}\in W^{\tilde P}.
\]
}
\end{lemma}

\begin{proof}
Note that $\gamma^{\vee}=\sum_{j=k}^{n}%
\alpha_{j}^{\vee}=s_{k}s_{k+1}\cdots s_{n-1}(\alpha_{n}^{\vee})$. We conclude
that $s_{\gamma}=s_{k}s_{k+1}\cdots s_{n}\cdots s_{k+1}s_{k}$ and
$\ell(s_{\gamma})=\langle2\rho, \gamma^{\vee}\rangle-1=2n-2k+1$.

Suppose assumption (a) holds first. Note that $\langle\alpha_{k}, \gamma
^{\vee}\rangle>0$. By Lemma \ref{lengthofpos222}, we have $\ell(vs_{\gamma
}s_{k})=\ell(vs_{\gamma})+1$. Hence, $vs_{\gamma}(\alpha_{k})\in R^{+}$. That
is, assumption (b) holds.

Assume (b) holds now. Note that $\langle\alpha_{i}, \gamma^{\vee}\rangle=0$
for any $\alpha_{i}\in\Delta\setminus\{\alpha_{k-1}, \alpha_{k}\}$, so that
$vs_{\gamma}(\alpha_{i})=v(\alpha_{i})\in R^{+}$. Hence, (c) follows.

Assume (c) holds, equivalently, $vs_{\gamma}(\alpha_{i})\in R^{+}$ for all
$\alpha_{i}\in\Delta\setminus\{\alpha_{k-1}\}$. Then we have $v\neq1$, because otherwise  $vs_{\gamma}%
(\alpha_{k})=\alpha_{k}-\langle\alpha_{k}, \gamma^{\vee}\rangle\gamma\notin
R^{+}$. As a consequence,
 we have $v(\alpha_{k})\in-R^{+}$ (as $v\in W^P$). Rewrite $(\alpha_{k}, \cdots, \alpha_{n}, \cdots, \alpha_{k})$  as
$(\beta_{1}, \cdots, \beta_{2n-2k+1})$.  To show (a), it suffices to show
$\ell(vs_{\gamma}s_{\beta_{1}}\cdots s_{\beta_{j-1}}s_{\beta_{j}})=
\ell(vs_{\gamma}s_{\beta_{1}}\cdots s_{\beta_{j-1}})+1$ (or equivalently to
show $vs_{\gamma}s_{\beta_{1}}\cdots s_{\beta_{j-1}}(\beta_{j})\in R^{+}$) for
all  $1\leq j\leq2n-2k+1$. Since $\alpha_{k}=\beta_{1}$, this holds when
$j=1$. When $2\leq j\leq2n-2k$, we note that  $s_{\beta_{1}}\cdots
s_{\beta_{j-1}}(\beta_{j})=\alpha_{k}+\beta$ for a positive root $\beta$ in the root
subsystem with respect to the subbase $\{\alpha_{k+1}, \cdots, \alpha_{n}\}$.
Thus $vs_{\gamma}s_{\beta_{1}}\cdots s_{\beta_{j-1}}(\beta_{j})=vs_{\gamma
}(\alpha_{k})+vs_{\gamma}(\beta)=vs_{\gamma}(\alpha_{k})+v(\beta)\in R^{+}$.
When $j=2n-2k+1$, we have $vs_{\gamma}s_{\beta_{1}}\cdots s_{\beta_{j-1}%
}(\beta_{j})=vs_{\gamma}s_{\gamma}s_{\beta_{2n-2k+1}}(\beta_{2n-2k+1}%
)=-v(\beta_{2n-2k+1})=-v(\alpha_{k})\in R^{+}$. Thus we are done.
\end{proof}

Thanks to the above lemma,  the assumption  ``$\ell(vs_{\gamma})=\ell(v)-2n+2k-1$" in Theorem \ref{thmQPRforIGIG} is equivalent to
the assumption ``$vs_{\gamma}\in W^{\tilde P}$". This indicates
us that $N_{u, v}^{w, 1}=N_{us_{k}, vs_{\gamma}}^{ws_{1}\cdots s_{k-1}, 0}$ is a classical
intersection number involved in the cup product $\sigma^{s_{k-p+1}\cdots
s_{k-1}}\cup\sigma^{vs_{\gamma}}$ in $H^{*}(IG(k-1, 2n))$. As a consequence,
the classical Pieri rule in \cite{pragrat} can still be applied. In
particular, we can reformulate Theorem \ref{thmQPRforIGIG} in a more
traditional way, which will be described in  Theorem \ref{thmQPRforIGIG222}.

\subsection{Classical aspects of quantum Pieri rules for Grassmannians of type $B, D$}

Throughout this subsection, we consider a Grassmannian of type $B_{n}$ or
$D_{n+1}$. Precisely, we consider  the isotropic Grassmannian $G/P=OG(k,
2n+1)$ (resp. $OG(k, 2n+2)$) for $\Delta$ of type $B_{n}$ (resp. $D_{n+1}$).
Note that the even orthogonal Grassmannian $OG^{o}(n+1, 2n+2)$ is isomorphic
to the odd orthogonal Grassmannian $OG(n, 2n+1)$. It suffices to deal with
either of them only. Hence, when $\Delta$ is of $D_{n+1}$-type, we can always
assume $k\leq n-1$. In other words, whenever referring to ``$k=n$", we are
dealing with  $\Delta$ of $B_{n}$-type, unless otherwise stated. As before, we
need to compute the  Gromov-Witten invariants $N_{u, v}^{w, d}=N_{u, v}^{w,
\lambda_{P}}=N_{u, v}^{\tilde w, \lambda_{B}}$ (where $\lambda_{P}=d\alpha
_{k}^{\vee}+Q_{P}^{\vee}$)  for the quantum product $\sigma^{u}\star\sigma
^{v}$ in $QH^{*}(G/P)$.

Let $[x]$ denote the integer satisfying $0\leq x-[x]<1$. In order to state the
results uniformly, we denote
\[
\bar\alpha_{n}^{\vee}:=\alpha_{n}^{\vee}\mbox{ (resp. } \alpha_{n}^{\vee
}+\alpha_{n+1}^{\vee})\quad\mbox{ and }\quad s_{\bar\alpha_{n}}:= s_{n}
\mbox{ (resp. } s_{n}s_{n+1})
\]
when $\Delta$ is of type $B_{n}$ (resp. $D_{n+1}$). Furthermore, we denote
$D=d$ (resp. $2d$) if $k<n$ (resp. $k=n$).

With the same arguments as for Lemma \ref{lemmalambdaBfortypeCC}, we have

\begin{lemma}
\label{lambdaBfortypeBBDD}  Write $D=mk+r$ where $m, r\in\mathbb{Z}$ with
$1\leq r\leq k$. Then we have
\[
\lambda_{B}=m\sum_{j=1}^{k-1}j\alpha_{j}^{\vee}+\sum_{j=1}^{r-1}%
j\alpha_{k-r+j}^{\vee}+  d\alpha_{k}^{\vee}+%
\begin{cases}
2[{\frac{d}{2}}]\sum_{j=k+1}^{n-1}\alpha_{j}^{\vee}+[{\frac{d}{2}}]\bar
\alpha_{n}^{\vee} & \mbox{if } k<n\\
0 & \mbox{if } k=n
\end{cases}
.
\]
Consequently, we have  $\langle\alpha_{k}, \lambda_{B}\rangle=m+1+D-2[{\frac
{D}{2}}]$; for $k+1\leq n$,   $\langle\alpha_{k+1}, \lambda_{B}\rangle=-d+2[{\frac
{d}{2}}]$; for   $1\leq i\leq k-1$,   $\langle\alpha_{i}, \lambda_{B}\rangle=-1$ if $i=k-r$, or $0$ otherwise.
\end{lemma}

Recall $u=s_{k-p+1}\cdots s_{k-1}s_k$ where $1\leq p\leq k$.

\begin{lemma}
\label{lemmavanish11forBD}  If $D\geq k+1$, then we have $N_{u, v}^{w, d}=0$
for any $w\in W^{P}$.
\end{lemma}

\begin{proof}
Use the same notations as in Lemma \ref{lambdaBfortypeBBDD}.  If $D>2k$, then
we have $\langle\alpha_{k}, \lambda_{B}\rangle\geq m+1>2$  and consequently
$N_{u, v}^{w, d}=0$ by Proposition \ref{propmainthminLeungLi}.  Now we assume
$k+1\leq D\leq2k$, so that $m=1$.

Suppose $k<n$, that is, $D=d$. Note that  $\langle\alpha_{k}, \lambda
_{B}\rangle=1+1+d-2[{\frac{d}{2}}]=3$ (resp. $2$) if $d$ is odd (resp. even).
Thus when $d$ is odd, we are already done. When $d$ is even, we note that
$\mbox{sgn}_{{k+1}}(us_{k})+\mbox{sgn}_{{k+1}}(v)=0<1\leq\mbox{sgn}_{{k+1}%
}(\tilde w s_{k})+\langle\alpha_{k+1}, \lambda_{B}-\alpha_{k}^{\vee}\rangle$.
Thus we have $N_{us_{k}, v}^{\tilde w s_{k}, \lambda_{B}-\alpha_{k}^{\vee}%
}=0$  by using Proposition \ref{propmainthminLeungLi}(1).  As a consequence,
we have $N_{u, v}^{\tilde w, \lambda_{B}}=0$, by using Corollary
\ref{corovanish} (1). That is, $N_{u, v}^{w, d}=0$.

Suppose $k=n$, that is, $D=2d$. Note that $\langle\alpha_{n}, \lambda
_{B}\rangle=2$, $\langle\alpha_{n-1}, \lambda_{B}\rangle=0$ and consequently
$\langle\alpha_{n-1}, \lambda_{B}-\alpha_{n}^{\vee}\rangle=2$. Since
$\mbox{sgn}_{{n-1}}(us_{n})+\mbox{sgn}_{{n-1}}(v)\leq1+0<2= \langle\alpha
_{n-1}, \lambda_{B}-\alpha_{n}^{\vee}\rangle$,  we have $N_{us_{n}, v}^{\tilde
ws_{n}, \lambda_{B}-\alpha_{n}^{\vee}}=0$ by Proposition \ref{propmainthminLeungLi} (1).
Thus we have  $N_{u, v}^{\tilde w, \lambda_{B}}=0$ by Corollary
\ref{corovanish} (1).
\end{proof}

\begin{lemma}
\label{lemvanforBBDDkknn}  Let $u^{\prime}, v^{\prime}, w^{\prime}\in W$ and
$\lambda\in Q^{\vee}$. For the quantum product  $\sigma^{u^{\prime}}%
\star\sigma^{v^{\prime}}$ in $QH^{*}(G/B)$, the structure constant
$N_{u^{\prime}, v^{\prime}}^{w^{\prime}, \lambda}$ vanishes, if
  both {\upshape (a)} and {\upshape (b)} hold:

\begin{enumerate}

\item[(a)] $\langle\alpha_{n-1}, \lambda\rangle=\langle\alpha_{n-2},
\lambda\rangle=0$; $\langle\alpha_{j}, \lambda\rangle=1$, \textit{whenever}
$j\geq n$.

\item[(b)] {\upshape $\mbox{sgn}_{i}(u^{\prime})=\mbox{sgn}_{i}(v^{\prime})=0$
\textit{for} $i\in\{n-2, n-1\}$; $\mbox{sgn}_{j}(u^{\prime}s_{n-1})=0$,
\textit{if} $j\geq n$.
 }
\end{enumerate}

\end{lemma}

\begin{proof}
Since $\mbox{sgn}_{n-2}(v^{\prime})=0$, we have $v^{\prime}s_{\bar\alpha_{n}%
}(\alpha_{n-2})=v^{\prime}(\alpha_{n-2})\in R^{+}$. Thus $\mbox{sgn}_{n-2}%
(v^{\prime}s_{\bar\alpha_{n}})=0$.  Consequently, we have $\mbox{sgn}_{n-2}%
(u^{\prime})+\mbox{sgn}_{n-2}(v^{\prime}s_{\bar\alpha_{n}})=0<1=\langle
\alpha_{n-2}, -\alpha_{n-1}^{\vee}\rangle=  \langle\alpha_{n-2}, \lambda
-\bar\alpha_{n}^{\vee}-\alpha_{n-1}^{\vee}\rangle$.  By Proposition
\ref{propmainthminLeungLi} (1), we have  $N_{u^{\prime},
v^{\prime}s_{\bar\alpha_{n}}}^{w^{\prime}s_{n-1}s_{\bar\alpha_{n}}s_{n-1},
\lambda-\bar\alpha_{n}^{\vee}-\alpha_{n-1}^{\vee}}=0$. Since $\langle
\alpha_{n-1}, \lambda-\bar\alpha_{n}^{\vee}\rangle=-\langle\alpha_{n-1},
\bar\alpha_{n}^{\vee}\rangle=2$,  we have $N_{u^{\prime}s_{n-1}, v^{\prime
}s_{\bar\alpha_{n}}}^{w^{\prime}s_{n-1}s_{\bar\alpha_{n}}, \lambda-\bar
\alpha_{n}^{\vee}}=0$  by Corollary \ref{corovanish} (1).
Note that $\mbox{sgn}_{j}(u^{\prime}s_{n-1})=0$ and $\langle\alpha_{j}, \lambda
-\bar\alpha_{n}^{\vee}+\alpha_{j}^{\vee}\rangle=1+  \langle\alpha_{j},
-\bar\alpha_{n}^{\vee}+\alpha_{j}^{\vee}\rangle=1$, whenever $j\geq n$.
Using Corollary \ref{corovanish} (2),  we deduce   $N_{u^{\prime}s_{n-1},
v^{\prime}}^{w^{\prime}s_{n-1}, \lambda}=0$.  Then  we have  $N_{u^{\prime}, v^{\prime}}^{w^{\prime}, \lambda}%
=0$  by  Corollary \ref{corovanish} (3).
\end{proof}

\begin{prop}
\label{propdlarger2BBDD}  If $D\geq3$, then we have $N_{u, v}^{w, d}=0$ for
any $w\in W^{P}$.
\end{prop}

\begin{proof}
We can assume $3\leq D\leq k$, due to Lemma \ref{lemmavanish11forBD}.

Suppose $k=n$. Then $D=2d$ and $\lambda_{B}= \sum_{j=1}^{2d-1}j\alpha
_{n-2d+j}^{\vee}+  d\alpha_{n}^{\vee}$. It is easy to check that all the
assumptions in Lemma \ref{lemvanforBBDDkknn} hold for  $u, v, \lambda_{B}$.
Thus we are done.

Suppose $k<n$ now. Then $D=d$. Recall that $u=s_{k-p+1}\cdots s_{k}$ with
$\ell(u)=p$.

Assume $d$ is odd. Then $\lambda_{B}=\sum_{j=1}^{d-1}j\alpha_{k-d+j}^{\vee}+
d\alpha_{k}^{\vee}+ (d-1)\sum_{j=k+1}^{n-1}\alpha_{j}^{\vee}+ {\frac{d-1}{2}%
}\bar\alpha_{n}^{\vee}$.  Consequently, we have $\langle\alpha_{k},
\lambda_{B}\rangle=2$ and $\langle\alpha_{j}, \lambda_{B}\rangle=0$ for each
$k-d+1\leq j\leq k-1$. Denote $\bar d:=\min\{p, d\}$, $u^{\prime}:=us_{k}
s_{k-1}\cdots s_{k-\bar d+1}$ and  $\lambda:=\lambda_{B}-\sum_{j=1}^{\bar d}
\alpha_{k-\bar d+j}^\vee$. We claim $N_{u^{\prime}, v}^{\tilde w s_{k}
s_{k-1}\cdots s_{k-\bar d+1}, \lambda}=0$.  \big(Indeed, if $p\leq d$, then
$u^{\prime}=1$ and therefore the claim follows, by noting $\lambda\neq0$.
Note that $p\leq k$. If $p>d$, then $\bar d=d$, $u^{\prime}=s_{k-p+1}%
s_{k-p+2}\cdots s_{k-d}$ and we note that  $\langle\chi_{k-d}, \lambda
\rangle=0$. Thus the claim still follows by Lemma \ref{lemforBBBCCCDD}.\big)
Note that $\mbox{sgn}_{k-\bar d+1}(v)=0$ and $\langle\alpha_{k-\bar d+1},
\lambda+\alpha_{k-\bar d+1}^{\vee}\rangle=1$.
Applying Corollary \ref{corovanish}
(2) to $(v, u's_{k-\bar d+1}, \tilde w s_ks_{k-1}\cdots s_{k-\bar d+2}, \lambda+\alpha_{k-\bar d +1}^\vee, \alpha_{k-\bar d +1})$,
we obtain $N_{u^{\prime}s_{k-\bar d+1}, v}^{\tilde w s_{k} s_{k-1}\cdots
s_{k-\bar d+2}, \lambda+\sum_{j=1}^{h} \alpha_{k-\bar d+j}^{\vee}}=0$ for
$h=1$.  By induction, we conclude $N_{u^{\prime}s_{k-\bar d+1}\cdots s_{k-\bar
d+h}, v}^{\tilde w s_{k} s_{k-1}\cdots s_{k-\bar d+h+1}, \lambda+\sum
_{j=1}^{h} \alpha_{k-\bar d+j}^{\vee}}=0$ for each $1\leq h\leq\bar d-1$. In
particular, we have  $N_{us_{k}, v}^{\tilde w s_{k}, \lambda_{B}-\alpha
_{k}^{\vee}}=0$ when $h=\bar d-1$.  Since $\langle\alpha_{k}, \lambda
_{B}\rangle=2$, we have $N_{u, v}^{\tilde w, \lambda_{B}}=0$ by Corollary
\ref{corovanish} (1).

Assume $d$ is even. Then $\lambda_{B}=\sum_{j=1}^{d-1}j\alpha_{k-d+j}^{\vee}+
d\sum_{j=k}^{n-1}\alpha_{j}^{\vee}+ {\frac{d}{2}}\bar\alpha_{n}^{\vee}$.
Consequently, we have $\langle\alpha_{k}, \lambda_{B}\rangle=1$ and
$\langle\alpha_{j}, \lambda_{B}\rangle=0$ for any $j\notin\{k, k-d\}$.  Using
exactly the same arguments as in the third paragraph of the proof of
Proposition \ref{propdlarger2CCC}, we conclude that  it suffices to show
$N_{u^{\prime}, v^{\prime}}^{w^{\prime}, \lambda_{B}^{\prime}}=0$, in order to
show $N_{u, v}^{\tilde w, \lambda_{B}}=0$.  Here $u^{\prime}=us_{k+1}\cdots
s_{n-1}s_{\bar\alpha_{n}}$,  $v^{\prime}=vv^{(k)}_{d}\cdots v^{(n-1)}_{d}$,
$w^{\prime}=\tilde ws_{k+1}\cdots s_{n-1}s_{\bar\alpha_{n}}v^{(k)}_{d}\cdots
v^{(n-1)}_{d}$ and $\lambda_{B}^{\prime}=\sum_{j=1}^{d-1}j\alpha_{n-d+j}%
^{\vee}+{\frac{d}{2}}\bar\alpha_{n}^{\vee}$, where  $v_{d}^{(i)}:=
s_{i}s_{i-1}\cdots s_{i-d+1}$ for any $k\leq i\leq n-1$.  Hence, we are done,
by using Lemma \ref{lemvanforBBDDkknn} with respect to $u^{\prime}, v^{\prime
}, \lambda_{B}^{\prime}$.  \big(Indeed, we have $v_{d}^{(i)}(\alpha
_{i})=\alpha_{i-1}$ and $v_{d}^{(i)}(\alpha_{i-1})=\alpha_{i-2}$ for each
$k\leq i\leq n-1$, by noting $d\geq3$.  Thus we have $v^{\prime}(\alpha
_{n-1})=v(\alpha_{k-1})\in R^{+}$, $v^{\prime}(\alpha_{n-2})=v(\alpha
_{k-2})\in R^{+}$ and consequently $\mbox{sgn}_{n-1}(v^{\prime}%
)=\mbox{sgn}_{n-2}(v^{\prime})=0$. It is easy to check that all the remaining
assumptions in Lemma \ref{lemvanforBBDDkknn} hold for $u^{\prime}, v^{\prime},
\lambda_{B}^{\prime}$.\big)
\end{proof}

\begin{thm}
[Classical aspects of quantum Pieri rules for Grassmannians of type $B_{n},D_{n+1}$]%
\label{thmQPRforOGOG} ${}$
Let $\sigma^{u}, \sigma^{v}$ be Schubert classes in the quantum cohomolgy of
the isotropic Grassmannian $G/P=OG(k, N)$, where $N=2n+1$ (resp. $2n+2$)  for
$\Delta$ of type $B_{n}$ (resp. $D_{n+1}$).  Recall  $u=s_{k-p+1}\cdots
s_{k-1}s_{k}$, where $1\leq p\leq k$. (Note that $c_{p}(\mathcal{S}%
^{*})=\sigma^{u}$, possibly up to a scale factor of $2$,  where $\mathcal{S}$
denotes the tautological subbundle over $OG(k, N)$.)  Then in the quantum
product
\[
\sigma^{u}\star\sigma^{v}=\sigma^{u}\cup\sigma^{v} + \sum_{w\in W^{P}, d\geq1}
N_{u, v}^{w, d}t^{d}\sigma^{w},
\]
all the Gromov-Witten invariants $N_{u, v}^{w, d}$ coincide with certain
classical intersection numbers.  More precisely, we have

\begin{enumerate}

\item If $d=1$, then we have
\[
N_{u, v}^{w, 1}=N_{u_{1}, v_{1}}^{w_{1}, 0}%
\]
with the elements $u_{1}, v_{1}, w_{1}\in W$ given by
\[
(u_{1}, v_{1}, w_{1})=%
\begin{cases}
(us_{k}, vs_{k}s_{k+1}\cdots s_{n-1}s_{\bar\alpha_{n}%
}s_{n-1}\cdots s_{k+1},   ws_{1}\cdots s_{k-1}) & \mbox{if } k<n\\
(u, vs_{n}s_{n-1}, ws_{2}\cdots s_{n-1}s_{1}\cdots s_{n-2} s_{n-1}s_{n}) &
\mbox{if } k=n
\end{cases}
,
\]
provided that $\ell(u)+\ell(v)=\ell(w)+\deg t$, and zero otherwise.

\item If $d=2$, then we have
\[
N_{u, v}^{w, 2}=N_{u, v_{2}}^{w_{2}, 0}%
\]
with  $ v_{2}%
=vs_{k}\cdots s_{n-1}s_{\bar\alpha_{n}} s_{n-1}\cdots s_1$
and  $w_{2}=ws_{1}\cdots s_{n-1} s_{\bar\alpha_{n}}%
s_{n-1}\cdots s_{k}$,  provided
that  $k<n$ and $\ell(u)+\ell(v)=\ell(w)+2\deg t$, and zero otherwise.

\item If $d\geq3$, then we have $N_{u, v}^{w, d}=0$.
\end{enumerate}

\end{thm}

\begin{proof}
Recall that we have $N_{u, v}^{w, d}=N_{u, v}^{w\omega_{P}\omega_{P^{\prime}},
\lambda_{B}}$, where $\omega_{P}\omega_{P^{\prime}}, \lambda_{B}$ are elements
associated to  $\lambda_{P}=d\alpha_{k}^{\vee}+Q^{\vee}_{P}$, defined by the
Peterson-Woodward comparison formula. Furthermore,  we have $N_{u, v}^{w,
d}=0$ unless $\ell(u)+\ell(v)=\ell(w)+d\cdot\deg t$, because of the dimension
constraint. By Proposition \ref{propdlarger2BBDD}, we have $N_{u, v}^{w, d}%
=0$, for either of the cases: (1)  $d\geq3$ and $k<n$; (2) $d\geq2$ and $k=n$.
For the remaining cases, we assume $d=1$ first.

When $k<n$, we have $\lambda_{B}=\alpha_{k}^{\vee}$ (by Lemma
\ref{lambdaBfortypeBBDD}) and consequently  $\Delta_{P^{\prime}}=\Delta
_{P}\setminus\{\alpha_{k-1}, \alpha_{k+1}\}$. Denote $v^{\prime}%
:=s_{k+1}\cdots s_{n-1}s_{\bar\alpha_{n}}s_{n-1}\cdots s_{k+1}$. By direct
calculations, we conclude that  $\ell(\omega_{P}\omega_{P^{\prime}}%
)=|R_{P}^{+}|-|R_{P^{\prime}}^{+}|=\ell(s_{1}\cdots s_{k-1} v^{\prime})$ and
$s_{1}\cdots s_{k-1} v^{\prime}(\alpha)\in R^{+}$ for all  $\alpha\in\Delta_{P^{\prime}}$. Thus we
have $\omega_{P}\omega_{P^{\prime}}=s_{1}\cdots s_{k-1}v^{\prime}$ by Lemma
\ref{charalongest}. 
 Furthermore by Proposition \ref{propmainthminLeungLi}(2), we
 have $N_{u, v}^{w, 1}=N_{u, v}^{w\omega_{P}\omega_{P^{\prime}},
\alpha_{k}^{\vee}}=N_{us_{k}, vs_{k}}^{w\omega_{P}\omega_{P^{\prime}}, 0}$.
Note that $v'=(v')^{-1}$ and $\mbox{sgn}_j(us_k)=0$ for all $j>k$.
It follows directly from Corollary \ref{coridentity} that
 $N_{us_{k}, vs_{k}}^{w\omega_{P}\omega_{P^{\prime}}, 0}=N_{us_k, vs_kv'}^{ws_1\cdots s_{k-1}, 0} $ if
 $\ell(vs_kv')=\ell(vs_k)-\ell(v')$, or $0$ otherwise.
In the latter case, we have   $N_{us_k, vs_kv'}^{ws_1\cdots s_{k-1}, 0} =0$ from     the dimension constraint  (by noting
 $\ell(us_k)+\ell(vs_k)=\ell(w)+  \deg t-2=\ell(w\omega_{P}\omega_{P^{\prime}})=\ell(ws_1\cdots s_{k-1})+\ell(v')$).
Thus (1) follows in this case.

When $k=n$ (in this case $\Delta$ is of $B_{n}$-type by our assumption), we
have $\lambda_{B}=\alpha_{n-1}^{\vee}+\alpha_{n}^{\vee}$, so that
$\Delta_{P^{\prime}}=\Delta_P\setminus\{\alpha_{n-2}\}$. Using Lemma \ref{lemmaprodAtype}, we obtain $\omega
_{P}\omega_{P^{\prime}}=s_{2}\cdots s_{n-1}s_{1}\cdots s_{n-2}$.
Consequently,  we have $N_{u, v}^{w, 1}=N_{u, v}^{w\omega_{P}\omega
_{P^{\prime}}, \lambda_{B}}=N_{u, v}^{w_{1}s_{n}s_{n-1}, \lambda_{B}}=:a_{1}$.
Furthermore, we note that $\ell(\omega_{P}\omega_{P^{\prime}})+\langle2\rho,
\lambda_{B}\rangle=2n=\deg t=\ell(u)+\ell(v)-\ell(w)\leq n+\ell(v)$. Thus
$\ell(v)\geq2$ and consequently $\mbox{sgn}_{n}(v)= \mbox{sgn}_{n-1}%
(vs_{n})=1$. Denote  $a_{2}:=N_{us_{n-1}, v}^{w_{1}s_{n}, \lambda_{B}}$,
$a_{3}:=N_{us_{n-1}, vs_{n}}^{w_{1}, \lambda_{B}-\alpha_{n}^{\vee}}$ and
$a_{4}:=N_{u, vs_{n}s_{n-1}}^{w_{1}, \lambda_{B}-\alpha_{n}^{\vee}%
-\alpha_{n-1}^{\vee}}$. We can show the following identities.

\begin{enumerate}

\item[i)] $a_{1}=a_{2}$. Indeed, if $a_{2}=0$, then we have $a_{1}=0$, by noting
$\mbox{sgn}_{n-1}(u)=\mbox{sgn}_{n-1}(v)=\langle\alpha_{n-1}, \lambda
_{B}\rangle=0$ and using Corollary \ref{corovanish} (3). If $a_{2}\neq0$, then
we have  $\mbox{sgn}_{n-1}(w_{1}s_{n})=\mbox{sgn}_{n-1}(us_{n-1}%
)+\mbox{sgn}_{n-1}(v)-\langle\alpha_{n-1}, \lambda_{B}\rangle=
\mbox{sgn}_{n-1}(us_{n-1})=1$. Thus $\mbox{sgn}_{n-1}(w_1s_ns_{n-1})=0$.
 Note
   $ \langle \alpha_{n-1}, \alpha_{n-1}^\vee+\lambda_B\rangle=2,   \mbox{sgn}_{n-1}(us_{n-1})=1$ and $\mbox{sgn}_{n-1}(vs_{n-1})=1.$
Applying ``$(u, v, w, \lambda, \alpha)$" in  equations \eqref{eq1} and  \eqref{eq2} of Proposition \ref{propmainthminLeungLi}(2) to $(us_{n-1}, vs_{n-1}, w_1s_ns_{n-1}, \alpha_{n-1}^\vee+\lambda_B,  \alpha_{n-1})$, we have   $N_{us_{n-1}, vs_{n-1}}^{w_1s_ns_{n-1}, \alpha_{n-1}^\vee+\lambda_B}=N_{us_{n-1}s_{n-1}, vs_{n-1}s_{n-1}}^{w_1s_{n}s_{n-1}, \lambda_B}=N_{u, v}^{w_1s_{n}s_{n-1}, \lambda_B}$ and
 $N_{us_{n-1}, vs_{n-1}}^{w_1s_ns_{n-1}, \alpha_{n-1}^\vee+\lambda_B}=N_{us_{n-1}, vs_{n-1}s_{n-1}}^{w_1s_{n}s_{n-1}s_{n-1}, (\alpha_{n-1}^\vee+\lambda_B)-\alpha_{n-1}^\vee}=N_{us_{n-1}, v}^{w_1s_{n}, \lambda_B}.$  Hence,
    $ N_{u, v}^{w_1s_{n}s_{n-1}, \lambda_B}=N_{us_{n-1}, v}^{w_1s_{n}, \lambda_B}$.
That is,   $a_1=a_2$.

\item[ii)] $a_{2}=a_{3}$. Indeed, if $a_{3}=0$, then we have $a_{2}=0$  by
Corollary \ref{corovanish} (2). If $a_{3}\neq0$,  then we have $\ell
(w_{1})+2=\ell(us_{n-1})+\ell(vs_{n})=\ell(u)+\ell(v)=\ell(w)+2n=\ell(w\omega_{P}%
\omega_{P^{\prime}})+4=\ell(w_{1}s_{n}s_{n-1})+4$.  Hence, $\ell(w_{1}%
s_{n}s_{n-1})=\ell(w_{1})-2$ and consequently we have $\mbox{sgn}_{n}%
(w_{1})=1$. Note that $\mbox{sgn}_{n}(us_{n-1})=\mbox{sgn}_{n}(vs_{n})=0$ and
$\langle\alpha_{n}, \alpha^{\vee}_{n-1}\rangle=-1$. Then we have $a_{2}=a_{3}$
by    Proposition \ref{propmainthminLeungLi} (2).

\item[iii)] $a_{3}=a_{4}$. If $\mbox{sgn}_{n-1}(w_{1})=0$, then we are done by
Proposition \ref{propmainthminLeungLi} (2).  If $\mbox{sgn}_{n-1}(w_{1})=1$,
then we have $a_{3}=0$ and $a_{4}=0$ by Proposition \ref{propmainthminLeungLi}
(1), so that we are also done.
\end{enumerate}

\noindent Hence, we have $a_{1}=a_{4}$. That is, $N_{u, v}^{w, 1}=N_{u_{1}, v_{1}%
}^{w_{1}, 0}$.

It remains to deal with the case when $d=2$ and $k<n$. In this case, we have
$\lambda_{B}= \alpha_{k-1}^{\vee}+ 2\sum_{j=k}^{n-1}\alpha_{j}^{\vee}+
\bar\alpha_{n}^{\vee}$, which satisfies $\langle\alpha, \lambda_{B}\rangle=0$
for all $\alpha\in\Delta_P\setminus\{\alpha_{k-2}\}$. By Lemma
\ref{charalongest} again, we conclude $\omega_{P}\omega_{P^{\prime}}%
=s_{2}\cdots s_{k-1}s_{1}\cdots s_{k-2}$ (by which we mean the unit $1$ if
$k-2\leq0$).
 Note
   $\langle \alpha_{k+1}, \alpha_{k+1}^\vee+\lambda_B\rangle=2, \mbox{sgn}_{k+1}(w\omega_P\omega_{P'})=0, \mbox{sgn}_{k+1}(us_{k+1})=1$ and  $\mbox{sgn}_{k+1}(vs_{k+1})=1.$
 Applying    ``$(u, v, w, \lambda, \alpha)$" in equations \eqref{eq1} and \eqref{eq2} of Proposition \ref{propmainthminLeungLi}(2) to $(us_{k+1}, vs_{k+1}, w\omega_P\omega_{P'}, \alpha_{k+1}^\vee+\lambda_B, \alpha_{k+1})$, we have
     $  N_{us_{k+1}, vs_{k+1}}^{w\omega_P\omega_{P'}, \alpha_{k+1}^\vee+\lambda_B}=N_{us_{k+1}s_{k+1}, vs_{k+1}s_{k+1}}^{w\omega_P\omega_{P'}, \alpha_{k+1}^\vee+\lambda_B-\alpha_{k+1}^\vee}=N_{u, v}^{w\omega_P\omega_{P'}, \lambda_B}  $
 and $  N_{us_{k+1}, vs_{k+1}}^{w\omega_P\omega_{P'}, \alpha_{k+1}^\vee+\lambda_B}=N_{us_{k+1}, vs_{k+1}s_{k+1}}^{w\omega_P\omega_{P'}s_{k+1}, \alpha_{k+1}^\vee+\lambda_B-\alpha_{k+1}^\vee}=N_{us_{k+1}, v}^{w\omega_P\omega_{P'}s_{k+1},  \lambda_B}.$
 Hence, we have   $   N_{u, v}^{w\omega_P\omega_{P'}, \lambda_B}=N_{us_{k+1}, v}^{w\omega_P\omega_{P'}s_{k+1},  \lambda_B}.$
 Then by  using  induction and the same arguments above, we conclude  $$ N_{u, v}^{w, 2}=N_{u, v}^{w\omega_{P}%
\omega_{P^{\prime}}, \lambda_{B}}
= N_{us_{k+1}\cdots
s_{n-1}s_{\bar\alpha_{n}}, v}^{w\omega_{P}\omega_{P^{\prime}}s_{k+1}\cdots
s_{n-1}s_{\bar\alpha_{n}}, \lambda_{B}}.$$

  Denote $(\beta_{1}, \cdots, \beta_{2n-2k}) :=(\alpha_{k}, \cdots, \alpha_{n-1}, \alpha_{k-1}, \cdots, \alpha_{n-2}) $ and set
$$\begin{array}{ll}
   u^{\prime\prime}:=us_{k+1}\cdots s_{n-1}s_{\bar\alpha_{n}}, &     v^{\prime\prime}:= vs_{\beta_{1}}\cdots s_{\beta_{2n-2k}},
  \\
     w^{\prime\prime}:=w\omega_{P}\omega_{P^{\prime}} s_{k+1}\cdots s_{n-1}s_{\bar\alpha_{n}}s_{\beta_{1}}\cdots s_{\beta_{2n-2k}}, &
       \lambda_{B}^{\prime\prime} :=\lambda_{B}-\sum_{i=1}^{2n-2k}\beta_{i}^{\vee}.
\end{array}$$
 Since    $v^{\prime\prime}(\alpha_{n-1})=
   v(\alpha_{k-1})\in R^{+}$, we have
$\mbox{sgn}_{n-1}(v^{\prime\prime})=0$.
Note that $\lambda_{B}^{\prime\prime} =\alpha_{n-1}^{\vee}+\bar\alpha_{n}^{\vee}$. 
  Applying    ``$(u, v, w, \lambda, \alpha)$" in equations \eqref{eq1} and \eqref{eq2} of Proposition \ref{propmainthminLeungLi}(2) to $(u''s_{n-1}, v''s_{n-1}, w'', \alpha_{n-1}^\vee+\lambda_B'', \alpha_{n+1})$, we conclude  $N_{u^{\prime\prime}, v^{\prime\prime}}%
^{w^{\prime\prime}, \lambda_{B}^{\prime\prime}}=N_{u^{\prime\prime}s_{n-1},
v^{\prime\prime}}^{w^{\prime\prime}s_{n-1}, \lambda_{B}^{\prime\prime}} $
 Applying    ``$(u, v, w, \lambda, \alpha)$" in equations \eqref{eq3} and \eqref{eq1} of Proposition \ref{propmainthminLeungLi}(2) to $(v'', u''s_{n-1}s_{\bar \alpha_n}$, $w''s_{n-1}s_{\bar \alpha_n}, \lambda_B'', \bar\alpha_{n})$ (where we use these inductions for twice, i.e., for   $\alpha_n$ and $\alpha_{n+1}$ respectively, in   the case of  type $D_{n+1}$), we conclude   $N_{u^{\prime\prime}s_{n-1},
v^{\prime\prime}}^{w^{\prime\prime}s_{n-1}, \lambda_{B}^{\prime\prime}} =N_{u^{\prime\prime}s_{n-1}, v^{\prime\prime}s_{\bar\alpha_{n}}}%
^{w^{\prime\prime}s_{n-1}s_{\bar\alpha_{n}}, \lambda_{B}^{\prime\prime}%
-\bar\alpha_{n}^{\vee}}.$
 Applying equation     \eqref{eq1} of Proposition \ref{propmainthminLeungLi}(2),  we conclude
$N_{u^{\prime\prime}s_{n-1}, v^{\prime\prime}s_{\bar\alpha_{n}}}%
^{w^{\prime\prime}s_{n-1}s_{\bar\alpha_{n}}, \lambda_{B}^{\prime\prime}%
-\bar\alpha_{n}^{\vee}}=    N_{u^{\prime\prime}, v^{\prime\prime}s_{\bar\alpha_{n}%
}s_{n-1}}^{w^{\prime\prime}s_{n-1}s_{\bar\alpha_{n}},0}.$
 (In order to apply  equations in Proposition \ref{propmainthminLeungLi}(2) above,   we have assumed the grading  equality ``$\mbox{sgn}_\alpha(u)+\mbox{sgn}_\alpha(v)=\mbox{sgn}_\alpha(w)+\langle\alpha, \lambda\rangle$" holds for all the following four structure constants. If  it  does not hold for any one of the structure constants, it is easy to show all of them  are equal to $0$, so that we always have the following equalities as expected.) That is, we have $$N_{u^{\prime\prime}, v^{\prime\prime}}%
^{w^{\prime\prime}, \lambda_{B}^{\prime\prime}}=N_{u^{\prime\prime}s_{n-1},
v^{\prime\prime}}^{w^{\prime\prime}s_{n-1}, \lambda_{B}^{\prime\prime}} =N_{u^{\prime\prime}s_{n-1}, v^{\prime\prime}s_{\bar\alpha_{n}}}%
^{w^{\prime\prime}s_{n-1}s_{\bar\alpha_{n}}, \lambda_{B}^{\prime\prime}%
-\bar\alpha_{n}^{\vee}}=    N_{u^{\prime\prime}, v^{\prime\prime}s_{\bar\alpha_{n}%
}s_{n-1}}^{w^{\prime\prime}s_{n-1}s_{\bar\alpha_{n}},0}.$$

Hence, (2) follows directly from Lemma \ref{lemmaforthmQPROG} and the next claim:
 $$ (\circledast)\qquad N_{us_{k+1}\cdots
s_{n-1}s_{\bar\alpha_{n}}, v}^{w\omega_{P}\omega_{P^{\prime}}s_{k+1}\cdots
s_{n-1}s_{\bar\alpha_{n}}, \lambda_{B}}=N_{u^{\prime\prime}, v^{\prime\prime}}^{w^{\prime\prime}, \lambda
_{B}^{\prime\prime}}.$$
 \big(Indeed if $N_{u^{\prime\prime},
v^{\prime\prime}}^{w^{\prime\prime}, \lambda_{B}^{\prime\prime}}=0$,  then we can show
  $N_{u^{\prime\prime}, v^{\prime\prime}s_{\beta_{2n-2k}}\cdots s_{\beta_{2n-2k-h+1}}}%
^{w^{\prime\prime}s_{\beta_{2n-2k}}\cdots s_{\beta_{2n-2k-h+1}}, \lambda_{B}^{\prime\prime}+\sum
_{i=1}^{h}\beta_{2n-2k-i+1}^{\vee}}$ $=0$ for $1\leq h\leq2n-2k$, by using Corollary
\ref{corovanish} (2) and   induction on $h$. In particular for $h=2n-2k$, we have
  $N_{u^{\prime\prime}, v^{\prime\prime}s_{\beta_{2n-2k}}\cdots
s_{\beta_{1}}}^{w^{\prime\prime}s_{\beta_{2n-2k}}\cdots s_{\beta_{1}},
\lambda_{B}}=0$, hence $(\circledast)$. 
 Since
$\ell(u)+\ell(v)=\ell(w)+2\deg t$,  we have $\ell(u^{\prime\prime}%
)+(v)=\ell(w\omega_{P}\omega_{P^{\prime}}s_{k+1}\cdots s_{n-1}s_{\bar
\alpha_{n}})+\langle2\rho, \lambda_{B}\rangle$. Thus if  $N_{u^{\prime\prime},
v^{\prime\prime}}^{w^{\prime\prime}, \lambda_{B}^{\prime\prime}}\neq0$, then
we have  $\ell(u^{\prime\prime})+\ell(v^{\prime\prime})=\ell(w^{\prime\prime
})+\langle2\rho, \lambda_{B}^{\prime\prime}\rangle$ and consequently
$\ell(u^{\prime\prime})+\ell(v^{\prime\prime}s_{\beta_{2n-2k}}\cdots
s_{\beta_{h}})=\ell(w^{\prime\prime}s_{\beta_{2n-2k}}\cdots s_{\beta_{h}})+
\langle2\rho, \sum_{i=h}^{2n-2k}\beta_{i}^{\vee}\rangle$ for all $1\leq
h\leq2n-2k$. Hence, we have $\langle\beta_{h}, \lambda_{B}-\sum_{i=1}^{h-1}\beta
_{i}^{\vee}\rangle=1$,  $\mbox{sgn}_{\beta_{h}}(w\omega_{P}\omega_{P^{\prime}%
}s_{k+1}\cdots s_{n-1}s_{\bar\alpha_{n}}s_{\beta_{1}}\cdots s_{\beta_{h-1}%
})=0$,  $\mbox{sgn}_{\beta_{h}}(u^{\prime\prime})$ $=0$ and  $\mbox{sgn}_{\beta
_{h}}(vs_{\beta_{1}}\cdots s_{\beta_{h-1}})=1$.  By  using
Proposition \ref{propmainthminLeungLi} (2) and induction on $h$, we   still conclude   that $(\circledast)$ holds.\big)
\end{proof}

\begin{lemma}\label{lemmaforthmQPROG}
Using the same assumptions and notations as in  Theorem \ref{thmQPRforOGOG} (and in the proof of it), we have
    $$ N_{u^{\prime\prime}, v^{\prime\prime}s_{\bar\alpha_{n}%
}s_{n-1}}^{w^{\prime\prime}s_{n-1}s_{\bar\alpha_{n}},0}=N_{u, v_2}^{w_2, 0}.$$
Furthermore if $N_{u, v_2}^{w_2, 0}\neq 0$, then we have
  $v_2\in W^P$ and  $\ell(v_2)=\ell(v)-\ell(v^{-1}v_2)$.
\end{lemma}
\begin{proof}
By Lemma \ref{lemmaprodAtype}, we  have
  $v'' s_{\bar\alpha_{n}}s_{n-1} s_{n-2}\cdots s_1=v_2\cdot (s_k s_{k+1}\cdots s_{n-1}) $ and
  $w''s_{n-1}s_{\bar\alpha_{n}}=ws_1\cdots s_{k-1}s_{k+1}\cdots s_{n-1}s_k\cdots s_{n-2} s_{\bar\alpha_{n}}s_{n-1}s_{\bar\alpha_{n}}s_1\cdots s_{n-2}$.
Denote  $w_3:=w''s_{n-1}s_{\bar\alpha_{n}}s_{n-2}\cdots s_1=ws_1\cdots s_{k-1}s_{k+1}\cdots s_{n-1}s_k\cdots s_{n-2} s_{\bar\alpha_{n}}s_{n-1}s_{\bar\alpha_{n}}$.

 Assume $N_{u^{\prime\prime}, v^{\prime\prime}s_{\bar\alpha_{n}%
}s_{n-1}}^{w^{\prime\prime}s_{n-1}s_{\bar\alpha_{n}},0}\neq 0$ first, then $\ell(u'')+\ell( v^{\prime\prime}s_{\bar\alpha_{n}%
}s_{n-1})=\ell(w^{\prime\prime}s_{n-1}s_{\bar\alpha_{n}})$. Observe that
    $\ell(u'')=\ell(u)+\ell(u^{-1}u'')$ and
  $\ell(w^{-1}w''s_{n-1}s_{\bar\alpha_{n}})+\ell(v^{-1}v''s_{\bar\alpha_{n}}s_{n-1})=2\deg t+\ell(u^{-1}u'')
       $. Combining the assumption $\ell(u)+\ell(v)=\ell(w)+2\deg t$, we conclude that
 both  $\ell(w''s_{n-1}s_{\bar\alpha_{n}})=\ell(w)+\ell(w^{-1}w''s_{n-1}s_{\bar\alpha_{n}})$ and
  $\ell(v''s_{\bar\alpha_{n}}s_{n-1})=\ell(v)-\ell(v^{-1}v''s_{\bar\alpha_{n}}s_{n-1})$ hold.
Furthermore by    Corollary \ref{coridentity}, we have
     $$N_{u'', v'' s_{\bar\alpha_{n}}s_{n-1}}^{w''s_{n-1}s_{\bar\alpha_{n}},0} =N_{u'', v'' s_{\bar\alpha_{n}}s_{n-1} s_{n-2}\cdots s_1}^{ w''s_{n-1}s_{\bar\alpha_{n}} \cdot s_{n-2}\cdots s_1 , 0}=N_{u'', v_2s_k s_{k+1}\cdots s_{n-1}}^{w_3, 0}\hspace{1.0cm}(*)$$  and \quad
       $\ell(v'' s_{\bar\alpha_{n}}s_{n-1} s_{n-2}\cdots s_1)=\ell(v'' s_{\bar\alpha_{n}}s_{n-1})-\ell(s_{n-2}\cdots s_1). \hspace{1.5cm}(*')$

\noindent Note
      $\ell(v)-\ell(v^{-1}v_2s_k s_{k+1}\cdots s_{n-1})\leq  \ell(v_2s_k s_{k+1}\cdots s_{n-1})$ $=\ell(v'' s_{\bar\alpha_{n}}s_{n-1} s_{n-2}\cdots s_1)
             $  $= \ell(v)-\ell(v^{-1}v'' s_{\bar\alpha_{n}}s_{n-1})- \ell(s_{n-2}\cdots s_1)=\ell(v)-\ell(v^{-1}v_2s_k s_{k+1}\cdots s_{n-1})$.
 Hence, the equality holds and   consequently   $ \ell(v_2s_k s_{k+1}\cdots s_{n-1})=\ell(v_2)- \ell(s_{k}s_{k+1}\cdots s_{n-1})$.
Then by Corollary \ref{coridentity} again, we have
      $$ N_{u'', v_2s_k s_{k+1}\cdots s_{n-1}}^{w_3, 0}
        = N_{u'', v_2}^{w_3s_{n-1}\cdots s_{k+1}s_k, 0}.\hspace{4cm} (**) $$
    Note  that
    $s_{\bar\alpha_{n}}    s_{n-1} s_{\bar\alpha_{n}}   s_{n-1}=  s_{n-1} s_{\bar\alpha_{n}}  s_{n-1} s_{\bar\alpha_{n}}$.
 We have  $w_3s_{n-1}\cdots s_{k+1}s_k= w_2s_{k+1}\cdots s_{n-1}  s_{\bar\alpha_{n}}$ with
               $\ell(w_2s_{k+1}\cdots s_{n-1}  s_{\bar\alpha_{n}})= \ell(w_2)+\ell(s_{k+1}\cdots s_{n-1}s_{\bar\alpha_{n}})$.
         Since $v\in W^P$, we have    $v_2(\alpha_j)\in R^+$ for all $j>k$.
   Thus by Corollary \ref{coridentity}, we have
    $$ N_{u'', v_2}^{w_3s_{n-1}\cdots s_{k+1}s_k, 0}= N_{us_{k+1}\cdots s_{n-1}s_{\bar\alpha_{n}}, v_2}^{w_2s_{k+1}\cdots s_{n-1}s_{\bar\alpha_{n}}, 0}=
       N_{u, v_2}^{w_2,  0}. \hspace{2.5cm} (***) $$
In particular, we have $N_{u, v_2}^{w_2,  0}=N_{u^{\prime\prime}, v^{\prime\prime}s_{\bar\alpha_{n}%
}s_{n-1}}^{w^{\prime\prime}s_{n-1}s_{\bar\alpha_{n}},0}\neq 0$.

Now we assume $N_{u, v_2}^{w_2, 0}\neq 0$, which implies  $\ell(u)+\ell(v_2)=\ell(w_2)$.  Note that $(w^{-1}w_2)^{-1}=v^{-1}v_2$ and $\ell(v^{-1}v_2)=\deg t$. Since
   $\ell(u)+\ell(v)=\ell(w)+2\deg t$,
   we have $\ell(w_2)=\ell(w)+\ell(w^{-1}w_2)$ and $\ell(v_2)=\ell(v)-\ell(v^{-1}v_2)$.
 Thus we have $\ell(v_2s_1)=\ell(v_2)+1$ and consequently  $v_2(\alpha_1)\in R^+$.
 Note  $v\in W^P$. It is easy to check that $v_2(\alpha)\in R^+$ for all $\alpha\in \Delta_P\setminus \{\alpha_1\}$.
 Hence, $v_2\in W^P$ and consequently $w_2\in W^P$. Thus $(***)$ follows directly from Corollary \ref{coridentity}.
 Then $(**)$ also follows from Corollary \ref{coridentity}, by noting $\ell(w_3s_{n-1}\cdots s_{k+1}s_k)=\ell(w_3)+\ell(s_{n-1}\cdots s_{k+1}s_k)$ and
     $N_{u'', v_2}^{w_3s_{n-1}\cdots s_{k+1}s_k, 0}\neq 0$.
 Furthermore,  we conclude that $(*)$ holds, by noting $(*')$ and using Corollary \ref{coridentity}.
  In particular, we have $ N_{u^{\prime\prime}, v^{\prime\prime}s_{\bar\alpha_{n}%
}s_{n-1}}^{w^{\prime\prime}s_{n-1}s_{\bar\alpha_{n}},0}=N_{u, v_2}^{w_2, 0}\neq 0$.

    If none of the above two assumptions holds, we still have  $ N_{u^{\prime\prime}, v^{\prime\prime}s_{\bar\alpha_{n}%
}s_{n-1}}^{w^{\prime\prime}s_{n-1}s_{\bar\alpha_{n}},0}=N_{u, v_2}^{w_2, 0}$, both of which vanish.
\end{proof}
\begin{remark}
Recall that the odd orthogonal Grassmannian $OG(n, 2n+1)$ is isomorphic to the
even orthogonal Grassmannian $OG^{o}(n+1, 2n+2)$. It suffices to deal with
either of them. The former case has been covered in the above theorem. The
later case has been dealt with earlier by Kresch and Tamvakis in
\cite{KreschTamvakisortho}.
\end{remark}

As indicated by Lemma \ref{lemmaforthmQPROG}, we can
use the classical Pieri rules given by Pragacz and Ratajski to
interpret the classical intersection numbers $N_{u, v_2}^{w_2, 0}$ explicitly.
For   $N_{u_1, v_1}^{w_1, 0}$,    when  $k=n$  we can make use of the
generalized classical Pieri rules  given by Bergeron and Sottile
 (see Theorem D of \cite{bergSottile}); when $k<n$,
we can use the classical Chevalley formula
for $k\in\{1, 2\}$.  However,  a   classical Pieri formula analogous with the one
given by   Bergeron and Sottile \cite{bergSottile} is still lacking in general. It will be desirable to derive such a formula.

\begin{remark}
 In our proof of   Theorem \ref{thmQPRforOGOG}, we make use of Proposition \ref{propmainthminLeungLi} to reduce
 $N_{u, v}^{w, 1}$ to   a  classical intersection number for the two step flag variety
   $OF(k-1, k+1; N)$ first. For this step, there is another approach    using the well-known fact that  the parameter space of lines
on the Grassmannian $OG(k, N)$ 
is      $OF(k-1,k+1;N)$,
 as pointed out explicitly by Buch, Kresch and Tamvakis in \cite{BKT2}.
\end{remark}

\section{Quantum Pieri rules for Grassmannians of classical types: reformulations in traditional ways}
\label{sectionreformalation}
Historically, the Schubert classes for complex
Grassmannians are labelled by partitions, and the (quantum) Pieri rule therein
expresses the  (quantum) product of a special partition and a general
partition. There are similar notions for isotropic Grassmannians. In this
 section, we  derive  our quantum Pieri rules for Grassmannians of type $C_n$ and $B_n$, by reformulating Theorem \ref{thmQPRforIGIG} and Theorem \ref{thmQPRforOGOG}
  in a   traditional way (i.e. in terms of ``partitions").  
In addition, we use the same notations as in section \ref{subsectisotroGrassm} and always assume $k<n$. The remaining cases will be discussed  in
 section \ref{subsecremarks}.

\subsection{Quantum Pieri rules for Grassmannians of type $C_n$}\label{subsecQPRCCC}
We first review some parameterizations of the minimal length representatives $W^P$ for the isotropic Grassmanniann $G/P=IG(k, 2n)$,  following
 \cite{pragrat} (see also \cite{tamvakis} and \cite{BKT2}).
Each element  $x$ in the permutation group $S_n$ is represented by its image $(x(1), \cdots, x(n))$, or simply $(x_1, \cdots, x_n)$.
The Weyl group $W$ of type $C_n$ is isomorphic to the hyperoctahedral group
$S_{n}\ltimes\mathbb{Z}_{2}^{n}$ of barred permutations, which is an extension
of the permutation group $S_{n}=\langle\dot s_{1}, \cdots, \dot s_{n-1}%
\rangle$ by an element $\dot s_{n}$ acting on the right by $(x_{1}, \cdots,
x_{n})\dot s_{n}=(x_{1}, \cdots, x_{n-1}, \bar x_{n})$. Here $\dot s_{i}$
denote the transposition $(i, i+1)$ for each $1\leq i\leq n-1$.
Each element $w$ in  $W^P$ can be identified with a sequence of the form $(y_1, \cdots, y_{k-m}, \bar z_m, \cdots, \bar z_{1}, v_1, \cdots, v_{n-k})$,
where $y_1<\cdots<y_{k-m}, z_m>\cdots >z_1$ and $v_1<\cdots <v_{n-k}$, as follows. Let $\epsilon_i:=-\mathbf{e}_i\in \{-1, 1\}\mathbf{e}_1\oplus\cdots\oplus \{-1, 1\}\mathbf{e}_n=\mathbb{Z}^n_2$ for each $1\leq i\leq n$. Write $w=u_1s_nu_2s_n\cdots u_js_nu_{j+1}$ where $u_j$'s are all in $S_n$. Denote $a_{i}:=
  (u_{i+1}u_{i+2}\cdots u_{j+1})^{-1}(n)$ for each $1\leq i\leq j$. Then $w\in W$ is identified with the barred permutation $(u_1u_2\cdots u_{j+1}, \epsilon_{a_1}\epsilon_{a_2}\cdots \epsilon_{a_j})\in S_{n}\ltimes\mathbb{Z}_{2}^{n}$. The inequalities among entries in the identified sequence automatically hold as a consequence of the property $w\in W^P$.
The element $w$ in   $W^P$ can also be identified  with
  an element  $\mu=(\mu
^{t}//\mu^{b})$ in the set $\mathcal{P}_{k}$ of \textit{shapes}. That is, $\mu^{t}$ and $\mu^{b}$ are strict partitions inside
$(n-k)$ by $n$ rectangle and $k$ by $n$ rectangle respectively, and they
satisfy the inequality $\mu_{n-k}^{t}\geq\ell(\mu^{b})+1$. Here  $\mu^{t}%
=(\mu_{1}^{t}, \cdots, \mu_{n-k}^{t})$ and $m:=\ell(\mu^{b})$ denotes the
length of the partition $\mu^{b}$.
Precisely, we have $\mu_j^b=n+1-z_j$ for each $1\leq j\leq m$, and $\mu^t_r=n+1-v_r+\sharp\{i~|~ z_i<v_r, i=1, \cdots, m\}$ for each $1\leq r\leq n-k$.
For  such a shape $\mu$ in
$\mathcal{P}_{k}$,  we have $|\mu|:=|\mu^{t}|+|\mu^{b}|-{\binom{n-k+1}{2}}%
$.  There is a particular reduced expression of the corresponding
$w=w_{\mu}\in W^{P}$ with $\ell(w_{\mu})=|\mu|$, given by
\begin{align*}
w_{\mu}= & (s_{n-\mu_{m}^{b}+1}\cdot s_{n-\mu^{b}_{m}+2}\cdot\cdots\cdot
s_{n-1}\cdot s_{n})\cdot\cdots\\
(\star)\qquad\quad& \cdot(s_{n-\mu_{1}^{b}+1}\cdot s_{n-\mu^{b}_{1}+2}\cdot\cdots\cdot
s_{n-1}\cdot s_{n})\cdot(s_{n-\mu_{n-k}^{t}+1}\cdot\cdots\cdot s_{n-2}\cdot
s_{n-1})\\
& \cdot(s_{n-\mu_{n-k-1}^{t}+1}\cdot\cdots\cdot s_{n-3}\cdot s_{n-2}%
)\cdot\cdots\cdot(s_{n-\mu_{1}^{t}+1}\cdot\cdots\cdot s_{k-1}\cdot s_{k}).
\end{align*}
In particular, for the special class $c_{p}(\mathcal{S}^{*})=\sigma^{u}$,
$u=s_{k-p+1}\cdots s_{k}=u_{\mu}$ corresponds to the special shape
$\mu=((n-k+p, n-k-1, \cdots, 1)//\emptyset)$.  Usually, such a special $\mu$  is   simply
denoted as $p\in\mathcal{P}_{k}$.
We note that the quantum
Pieri rules with respect to the Chern classes $c_{i}(\mathcal{Q})
$ (where $1\leq i\leq2n-k$)  have been given by Buch,
Kresch and Tamvakis \cite{BKT2}.
\begin{remark}
In terms of notations in
\cite{tamvakis},  $c_{i}(\mathcal{Q}%
)=\sigma^{u^{\prime}}$ (up to a scale factor of $2$) with $u^{\prime}$  corresponding to
the special shape $(1^{\min(i, n-k)}|\max(i-n+k, 0))$. In general, the notation of shapes $(\mathbf{a} | \mathbf{b})$ in \cite{tamvakis} 
          is slightly different from the notation $(\mu^t//\mu^b)$ in \cite{pragrat}. For the same $w\in W^P$, we have $\mathbf{b}=\mu^b$ and
           $\mathbf{a}=(a_1, \cdots, a_{n-k})$ with $a_r=\mu^t_r+r-n+k-1$ for each $1\leq r\leq n-k$. In addition, $w$ can also be identified with the
            $(n-k)$-strict partition  $\mathbf{a}'+\mathbf{b}$ introduced in \cite{BKT2}, where $\mathbf{a}'$ is the transpose of $\mathbf{a}$.
\end{remark}

Let $\tilde P$ denote the standard   parabolic subgroup that corresponds to the subset $\Delta\setminus\{\alpha_{k-1}\}$.
  The minimal length representatives $W^{\tilde P}$
 are identified with  shapes in $\mathcal{P}%
_{k-1}$. (Note   $G/\tilde P=IG(k-1, 2n)$ in this subsection.)  To reformulate Theorem \ref{thmQPRforIGIG} and Theorem \ref{thmQPRforOGOG} in terms of shapes,
 we need the following lemma, which tells us about the explicit identifications. 

\begin{lemma}
\label{lemmaforPieriidenti} 
\begin{enumerate}

\item  $s_{k-p+1}\cdots s_{k-1}$  corresponds to the special shape $p-1\in\mathcal{P}_{k-1}$.

\item Let  $v\in W^P$ correspond  to  $\mathbf{a}=({\mathbf{a}}^{t}//{\mathbf{a}}^{b})\in \mathcal{P}_k$.
  Denote  $m:=\ell(\mathbf{a}^{b})$,  $\gamma:=\alpha_n+2\sum_{j=k}^{n-1}\alpha_j$,
 $v_1:=vs_\gamma s_k$ and
    $ v_{2}:=vs_{k}\cdots s_{n-1}s_{{n}} s_{n-1}\cdots s_1$.
  \begin{enumerate}
    \item  $\ell(vs_{\gamma})=\ell(v)-2n+2k-1$   if and only if $a_1^b\geq a_1^t$. Furthermore when this holds,
  $vs_{\gamma}\in W^{\tilde P}$ corresponds to the shape
   $$\tilde{\mathbf{a}}%
=((a_{1}^{b}, a_{1}^{t}-1,
a_{2}^{t}-1, \cdots, a_{n-k}^{t}-1) //(a_{2}^{b},
a_{3}^{b}, \cdots, a_{ m}^{b}))\in \mathcal{P}_{k-1}.$$

  \item $\ell(v_1)=\ell(v)-2n+2k$   if and only if $a_1^b\geq a_2^t$. Furthermore  if this holds and  $\ell(vs_\gamma)=\ell(v_1)+1$,
        then  $a_1^t>a_1^b$ and  $v_1\in W^{\tilde P}$  corresponds to {} 
              $$\bar{\mathbf{a}}= ((a_{1}^{t}, a_{1}^{b}, a_{2}^{t}-1, \cdots, a_{n-k}^{t}-1) //(a_{2}^{b}, a_{3}^{b}, \cdots, a_{ m}^{b}))\in \mathcal{P}_{k-1}.$$

   \item  If $v_2\in W^P$ and  $\ell(v_2)=\ell(v)-\ell(v^{-1}v_2)$, then $a_1^b=n$ and $v_2$  corresponds to the shape
        $$\hat{\mathbf{a}}=((a_{1}^{t}-1, a_{2}^{t}-1, \cdots, a_{n-k}^{t}-1) //(a_{2}^{b}, a_{3}^{b}, \cdots, a_{ m}^{b}))\in \mathcal{P}_{k}.$$
 \end{enumerate}

\item Let $w\in W^P$ correspond to   $\mathbf{c}=({\mathbf{c}}^{t}//{\mathbf{c}}^{b})\in \mathcal{P}_k$. Denote $w_1:=ws_1\cdots s_{k-1}$.
   \begin{enumerate}
     \item $w_1 \in W^{\tilde P}$ if and only if $c_1^t<n$. Furthermore when this holds,
                   $w_1$ corresponds to the shape
    $$\tilde
{\mathbf{c}}=((n, c_{1}^{t}, c_{2}^{t},
\cdots, c_{n-k}^{t})// {\mathbf{c}}^{b})\in \mathcal{P}_{k-1}.$$ 

     \item  If $w_1s_k \in W^{\tilde P}$ if and only if  $c_1^t=n$, $c_2^t\leq n-2$. Furthermore when this holds,  $w_1s_k$ corresponds to
    $$\bar
{\mathbf{c}}=((n, n-1, c_{2}^{t},
\cdots, c_{n-k}^{t})// {\mathbf{c}}^{b})\in \mathcal{P}_{k-1}.$$ 

     \end{enumerate}

\item  Let $w'\in W^P$ correspond to   $\mathbf{d}=({\mathbf{d}}^{t}//{\mathbf{d}}^{b})\in \mathcal{P}_k$. Denote   $m':=\ell(\mathbf{d}^b)$ and
  $w_{2}:=w's_{1}\cdots s_{n-1} s_{{n}}s_{n-1}\cdots s_{k}$.
     If $w_2\in W^P$ and  $\ell(w_2)=\ell(w')+\ell(s_{1}\cdots s_{n-1} s_{{n}}s_{n-1}\cdots s_{k})$, then  $d_1^b<n$ and $w_2$  corresponds to the shape
      $$\hat{\mathbf{d}}=((d_{1}^{t}+1, d_{2}^{t}+1, \cdots, d_{n-k}^{t}+1) //(n, d_1^b, d_{2}^{b},  \cdots, d_{ m'}^{b}))\in \mathcal{P}_{k}.$$

 \end{enumerate}

\end{lemma}

\begin{proof}
    Clearly, statement (1)  follows.

Note that  for every case in Lemma \ref{lemmaprodAtype},   the expression on the right-hand side of the equality  is reduced.
     Furthermore, we note that
  $\ell(s_\gamma)=2n-2k+1$ and $vs_\gamma=vs_{k}s_{k+1}\cdots s_{n-1}s_{n}s_{n-1}\cdots s_{k+1}s_{k}=v_1s_k$.
Hence, $\ell(vs_\gamma)=\ell(v)-\ell(s_\gamma)$ holds if and only if
   both $\ell(vs_ks_{k+1}\cdots s_n)=\ell(v)-\ell(s_ks_{k+1}\cdots s_n)$ and $\ell(vs_\gamma)=\ell(vs_ks_{k+1}\cdots s_n)-\ell(s_{n-1}\cdots s_{k+1}s_k)$ hold.
Using $(\star)$ and Lemma \ref{lemmaprodAtype},  we have
 \begin{align*}
vs_\gamma=&(vs_ks_{k+1}\cdots s_n)\cdot (s_{n-1}\cdots s_{k+1}s_k) \\
      = & \big((s_{n-a_{ m}^{b}+1} \cdots
s_{n-1}  s_{n})\cdots (s_{n-a_{2}^{b}+1}  \cdots
s_{n-1}  s_{n})(s_{n-a_{1}^{b}+1}  \cdots
s_{n-2}s_{n-1})\\
&\cdot(s_{n-a_{n-k}^{t}+1} \cdots  s_{n-3}s_{n-2})\cdots(s_{n-a_{1}^{t}+1} \cdots  s_{k-2}s_{k-1}) \big)\cdot (s_{n-1}\cdots s_{k+1}s_{k})\\
       = & (s_{n-a_{ m}^{b}+1} \cdots
s_{n-1}  s_{n})\cdots (s_{n-a_{2}^{b}+1}  \cdots
s_{n-1}  s_{n})\cdot (s_{n-a_{n-k}^{t}+2}  \cdots
s_{n-2}s_{n-1})\\
&\cdot(s_{n-a_{n-k-1}^{t}+2} \cdots  s_{n-3}s_{n-2})\cdots(s_{n-a_{1}^{t}+2} \cdots  s_{k-1}s_{k}) \cdot (s_{n-a_1^b+1}\cdots s_{k-2}s_{k-1}),
\end{align*}
 the right-hand side of the last equality in which
   gives a reduced expression of $vs_\gamma$.
In other words, we have $n-a_1^b+1\leq \min\{n-a_{n-k}^t+1, \cdots, n-a_2^t+1, n-a_1^t+1, n\}=n-a_1^t+1$. That is, $a_1^b\geq a_1^t$. Since
  $\mathbf{a}\in\mathcal{P}_k$, we have
 $\tilde{\mathbf{a}}=((a_{1}^{b}, a_{1}^{t}-1,
a_{2}^{t}-1, \cdots, a_{n-k}^{t}-1)$ $//(a_{2}^{b},
a_{3}^{b}, \cdots, a_{ m}^{b}))\in \mathcal{P}_{k-1}$,  corresponding to $vs_\gamma$. Thus statement (2a) holds.

The remaining parts are also consequences of the formula $(\star)$ and  Lemma \ref{lemmaprodAtype}.   The arguments for them  are also similar.
 \end{proof}

For any shapes $\mu, \nu\in\mathcal{P}_{k}$, we denote by $e_k(\mu, \nu)$ 
 the cardinality of  the
set of \textit{components that are not extremal, not related and have no
$(\nu-\mu)$-boxes over them}. The relevant notions, together with the notion
of ``\textit{$\nu$ compatible with $\mu$}", can be found on page 152 and page
153 of \cite{pragrat}.  We always skip the subscription part of  $e_k(\mu, \nu)$, whenever    $e(\mu, \nu)$ is well understood.
    In addition, we denote the Schubert cohomology class
$\sigma^{w_{\mu}}$ as $\sigma^{\mu}$. The following   classical Pieri rule was given by
  Pragacz and Ratajski.

\begin{prop}[Theorem 2.2 of \cite{pragrat}]
  \label{propCPRCCC}
  For every  $\mathbf{a}\in\mathcal{P}_{k}$ and $p\leq k$, we have
\[
\sigma^{p}\star\sigma^{\mathbf{a}}=\sum2^{e(\mathbf{a}, \mathbf{b})}%
\sigma^{\mathbf{b}},
\]
where the  sum is over all shapes $\mathbf{b}\in\mathcal{P}_{k}$ compatible with
$\mathbf{a}$ such that $|\mathbf{b}|=|\mathbf{a}|+p$.
\end{prop}

Combining  the above proposition and (statement (1), (2a), (3a) of) Lemma \ref{lemmaforPieriidenti},
we can reformulate Theorem \ref{thmQPRforIGIG}  directly as follows.
\begin{thm}
[Quantum Pieri rule for $IG(k, 2n)$]\label{thmQPRforIGIG222}  Let $
\mathbf{a}\in\mathcal{P}_{k}$ and $p\leq k$. Using the same notations as in Lemma
\ref{lemmaforPieriidenti}, we have
\[
\sigma^{p}\star\sigma^{\mathbf{a}}=\sum2^{e(\mathbf{a}, \mathbf{b})}%
\sigma^{\mathbf{b}} + t\sum2^{e(\tilde{\mathbf{a}}, \tilde{\mathbf{c}})}%
\sigma^{\mathbf{c}},
\]
where the first summation is over all shapes $\mathbf{b}\in\mathcal{P}_{k}$
with $|\mathbf{b}|=|\mathbf{a}|+p$ such that $\mathbf{b}$ is compatible with
$\mathbf{a}$,  and the second summation is over all shapes $\mathbf{c}%
\in\mathcal{P}_{k}$ with $|\mathbf{c}|=|\mathbf{a}|+p-2n+k-1$ such that
$\tilde{\mathbf{c}}\in\mathcal{P}_{k-1}$ is compatible with $\tilde{\mathbf{a}}%
\in\mathcal{P}_{k-1}$. (Here we denote the second sum as $0$ if $\tilde
{\mathbf{a}}\not \in \mathcal{P}_{k-1}$.)
\end{thm}

\subsection{Quantum Pieri rules for Grassmannians of type $B_n$}

The Weyl group $W$ of type $B_n$ is also isomorphic to the hyperoctahedral group
$S_{n}\ltimes\mathbb{Z}_{2}^{n}$ of barred permutations. 
The minimal length representatives $W^P$ for the isotropic Grassmanniann $G/P=OG(k, 2n+1)$ can also be
identified with shapes in $\mathcal{P}_k$ as well as other parameterizations   described for $IG(k, 2n)$ in the same way as in section \ref{subsecQPRCCC}.
Therefore we can use all the same notations as in section \ref{subsecQPRCCC} but  replace   $IG(k, 2n)$ by  $OG(k, 2n+1)$.

In this subsection, we obtain our quantum Pieri rule  for $OG(k, 2n+1)$ 
(which may involve signs in some cases), by reformulating part of Theorem \ref{thmQPRforOGOG}.
For any shapes $\mu, \nu\in\mathcal{P}_{k}$, we denote by  $e'(\mu, \nu)$  the cardinality of the
set of \textit{components that are not  related and have no
$(\nu-\mu)$-boxes over them}.  The following  classical Pieri rule for  $OG(k, 2n+1)$ was also given by  Pragacz and Ratajski.
\begin{prop}[Theorem 10.1 of \cite{pragrat}]
  \label{propCPRBBB}
  For every  $\mathbf{a}\in\mathcal{P}_{k}$ and $p\leq k$, we have
\[
\sigma^{p}\star\sigma^{\mathbf{a}}=\sum2^{e'(\mathbf{a}, \mathbf{b})}%
\sigma^{\mathbf{b}},
\]
where the  sum is over all shapes $\mathbf{b}\in\mathcal{P}_{k}$ compatible with
$\mathbf{a}$ such that $|\mathbf{b}|=|\mathbf{a}|+p$.
\end{prop}

\begin{remark}
   The rational cohomology of the complete flag varieties of type $B_n$ and $C_n$ are isomorphic to each other. As a consequence, there is a  relationship  between the
   classical intersection numbers for these two flag varieties, which was explicitly described in section 3 of \cite{bergSottile} (see also section 2.2 of \cite{BKT2}).
   This provides one way to obtain the information on classical intersection numbers for   $OG(k, 2n+1)$ from
   that for $IG(k, 2n)$.
\end{remark}

Recall that   the minimal length representatives $W^{\tilde P}$
are identified with  shapes in $\mathcal{P}_{k-1}$. (Note   $G/\tilde P=OG(k-1, 2n+1)$ in this subsection.)

\begin{defn}
 Let $\{\mathbf{e}_1, \cdots, \mathbf{e}_n\}$ denote
      the canonical basis of $\mathbb{Z}^{n}$ and  $\vec{\mu}\in \mathbb{Z}^{n}$ denote the canonical vector associated to a given  $\mu=(\mu^t//\mu^b)\in \mathcal{P}_{k-1}$; that is,
     $\vec{\mu}=\sum_{i=1}^{n-k+1}\mu_i^t\mathbf{e}_i+ \sum_{j=1}^m \mu_j^b\mathbf{e}_{n-k+1+j}$ where $m:=\ell(\mu^b)$. We define the sets $S(\mu), \Gamma_1(\mu)$ and
      $\Gamma_2(\mu)$ associated to $\mu$ as follows.
    \begin{align*} S(\mu):=&\{\nu\in\mathcal{P}_{k-1}~|~ \vec{\nu}=\vec{\mu}-\mathbf{e}_{j} \mbox{ for some } 2\leq j\leq n-k+1+m\}\bigcup\\
                   &\left\{\nu\in\mathcal{P}_{k-1}\Bigg| \begin{cases}\vec{\nu}=\vec{\mu}-(f+1)\mathbf{e}_{n-k+1+j}+f\mathbf{e}_{i}\\
                          \mu_i^t=\mu_j^b+j-f-1\end{cases} \hspace{-0.4cm}\mbox{for some} \begin{cases}
                     f>0,\, 1\leq j\leq m\\ 1\leq i\leq n-k+1 \end{cases}\hspace{-0.2cm}\right\}; \\
   \Gamma_2(\mu):=&\left\{\nu\in\mathcal{P}_{k-1}~\Bigg|~ \exists 1\leq j\leq m  \mbox{ such that } \begin{cases}\vec{\nu}=\vec{\mu}-\mathbf{e}_{n-k+1+j}\\
                                        \mu^i_t\neq \mu_j^t+j-1, \,\,\forall \,2\leq i\leq n-k+1\end{cases}\right\}\\
                 &\bigcup\left\{\nu\in\mathcal{P}_{k-1}\Bigg| \begin{cases}\vec{\nu}=\vec{\mu}-(f+1)\mathbf{e}_{n-k+1+j}+f\mathbf{e}_{1}\\
                          \mu_1^t=\mu_j^b+j-f-1\end{cases} \hspace{-0.4cm}\mbox{for some }
                     f>0,\, 1\leq j\leq m \hspace{-0.0cm}\right\}.
       \end{align*}
   In addition, we note  $\Gamma_2(\mu)\subset S(\mu)$ and denote $\Gamma_1(\mu): = S(\mu)\setminus \Gamma_2(\mu)$.
\end{defn}

\begin{lemma}
\label{lemmafordeg1ChevallBBB}
Suppose $v', w'\in W^{\tilde P}$ correspond to the shapes $\nu, \mu\in \mathcal{P}_{k-1}$ respectively;
 then $N_{s_k, v'}^{w', 0}\neq 0$ if and only if  $\nu\in S(\mu)$. Furthermore, $N_{s_k, v'}^{w', 0}=1$  if $\nu\in \Gamma_1(\mu)$,  or $2$ if $\nu\in \Gamma_2(\mu)$.
\end{lemma}
\begin{proof}
  By the classical Chevalley formula (i.e. the classical part of Proposition \ref{quanchevalley}), we have
     $N_{s_k, v'}^{w', 0}\neq 0$  only if $w'=v's_{\gamma'}$ for some positive root $\gamma'$ satisfying $\ell(v's_{\gamma'})=\ell(v')+1$, and when this holds, we have $N_{s_k, v'}^{w', 0}= \langle\chi_k, (\gamma')^\vee\rangle$.
Thus $v'$ is obtained by deleting a unique simple reflection from the reduced expression of $w'$ that is given by $(\star)$ (with respect to $k-1$). Furthermore, such an induced expression of $v'$ is reduced.
Using Lemma \ref{lemmaprodAtype} (together with the definition  of $\mu$ as a shape in $\mathcal{P}_{k-1}$), we can derive another reduced expression of $v'$ in the form of $(\star)$, from the induced reduced expression of $v'$.
In particular, we conclude  that $\nu\in S(\mu)\cup\{((\mu_1^t-1, \mu_2^t, \cdots, \mu_{n-k+1}^t)//\mu^b)\}$ (details for which are similar to the proof
 of statement a) of Proposition 3.4 of \cite{leungli33}).

  Note that if $v'$ is obtained by deleting a simple reflection $s_j$ from the $\ell$th position, then we have $\gamma'=x^{-1}(\alpha_j^\vee)$ where
    $x$ is the product of simple reflections from the
   $(\ell+1)$th position to the end of the reduced expression of $w'$.
Assume $\nu\in \Gamma_2(\mu)$.
 If  $\vec{\nu}=\vec{\mu}-\mathbf{e}_{n-k+1+j}$ for some $1\leq j\leq m$,  then   $v'$ is obtained by deleting the simple reflection $s_{n-\mu_j^b+1}$.
 \begin{enumerate}
    \item If  $n-\mu_j^b+1=n$, then  $\mu_j^b=1$. Consequently, we have    $j=m$ and
           $(\gamma')^\vee=(w')^{-1}s_n(\alpha_n^\vee)=\alpha_n^\vee+2\sum_{i=k-1}^{n-1}\alpha_i^\vee$. Thus   $\langle\chi_k,  (\gamma')^\vee \rangle=2.$
    \item  If $n-\mu_j^b+1\neq n$, then it is less than $n$ and   $(\gamma')^\vee= x^{-1}(\alpha_n^\vee)=\gamma_{1}^\vee$, where for each $1\leq i\leq n-k+1$ we denote by $\gamma^\vee_i$ the following coroot
  $$(s_{i+k-2} \cdots s_{n-\mu_i^t+1})\cdots(s_{n-1}\cdots s_{n-\mu_{n-k+1}^t+1})\big(\alpha_n^\vee+2\sum_{h=n-j+1}^{n-1}\alpha_h^\vee+\sum_{h=n-\mu_j^b+2-j}^{n-j}\alpha_h^\vee\big).$$
  Denote $\mu_{n-k+2}^t:=1$ and  $\gamma_{n-k+2}^\vee:=\alpha_n^\vee+2\sum_{h=n-j+1}^{n-1}\alpha_h^\vee+\sum_{h=n-\mu_j^b+2-j}^{n-j}\alpha_h^\vee$.
Note  that $\mu^t$ is a strict partition.
  Since $\nu\in \Gamma_2(\mu)$,     $n-\mu_i^t+1\neq  n-\mu_j^b+2-j$ for all $2\leq i\leq n-k+1$;
consequently,  there exists a unique $2\leq r\leq n-k+2$ such that
        $(n-\mu_i^t+1)-(n-\mu_j^b+2-j)$ is positive if  $i\geq r$, or negative if $2\leq i< r$.
Note that $k+i-j-1\geq k+i-m-1\geq k+i-\mu_{n-k+1}^t=n-(\mu_{n-k+1}^t+n-k+1-i)+1\geq n-\mu^t_i+1> n-\mu_j^b+2-j$,
   for each $ n-k+1\geq i\geq r$.
  By induction on $i$ (descendingly), we conclude $\gamma_i^\vee=
              \alpha_n^\vee+2\sum_{h=k+i-j-1}^{n-1}\alpha_h^\vee+\sum_{h=n-\mu_j^b+2-j}^{k+i-j-2}\alpha_h^\vee$ for all  $n-k+2\geq i\geq r$.
    In particular, we obtain $\gamma_r^\vee$.
  Furthermore,  for each $r\geq i\geq 2$, we have    $n-\mu_j^b+1+(i-r)>n-\mu_{r-1}^t+1+(i-r)=n-(\mu_{r-1}^t+r-i)+1\geq n-{\mu_{i-1}^t}+1$.
   Thus, by induction on $i$ (descendingly), we conclude $\gamma_i^\vee=
              \alpha_n^\vee+2\sum_{h=k+i-j-1}^{n-1}\alpha_h^\vee+\sum_{h=n-\mu_j^b+2-j+i-r}^{k+i-j-2}\alpha_h^\vee$ for  all $r\geq i\geq 2$.
   In particular, we obtain $\gamma_2^\vee$, for which we note $k+2-j-1\leq k$. Thus
    $$  \langle\chi_k, (\gamma')^\vee\rangle = \langle\chi_k, s_{k-1}s_{k-2}\cdots s_{n-\mu_1^t+1}(\gamma_2^\vee)\rangle
              =\langle\chi_k,  \gamma_2^\vee \rangle=2.$$

 \end{enumerate}
 Otherwise, we have $\vec{\nu}=\vec{\mu}-(f+1)\mathbf{e}_{n-k+1+j}+f\mathbf{e}_{1}$ for some $f>0$
            In this case, $v'$ is obtained by deleting $s_{n-\mu^b_j+f+1}$, and
        we have  $(\gamma')^\vee=(\gamma_1')^\vee$  with $(\gamma_{n-k+2}')^\vee=
           \alpha_n^\vee+2\sum_{h=n-j+1}^{n-1}\alpha_h^\vee+\sum_{h=n-\mu_j^b+2-j+f}^{n-j}\alpha_h^\vee$ and
        $(\gamma_i')^\vee=(s_{i+k-2}s_{i+k-3}\cdots s_{n-\mu_i^t+1})((\gamma_{i+1}')^\vee)$ for each $n-k+1\geq i\geq 1$.
     Note that $k+i-j-1\geq n-\mu_i^t+1>n-\mu_1^t+1=n-\mu_j^b+2-j+f$ for each $n-k+1\geq i\geq 2$.
     Using the same arguments as for case $(2)$, we can show
  $  \langle\chi_k, (\gamma')^\vee\rangle = \langle\chi_k, s_{k-1}s_{k-2}\cdots s_{n-\mu_1^t+1}(\gamma_2'^\vee)\rangle
              =\langle\chi_k,  \gamma_2'^\vee \rangle=2.$

  Hence, if $\nu\in \Gamma_2(\mu)$, then we have
   $  N_{s_k, v'}^{w', 0}= \langle\chi_k, (\gamma')^\vee\rangle =2$.
  Similarly, we can show
   $\langle\chi_k, (\gamma')^\vee\rangle =1$ if $\nu\in \Gamma_1(\mu)$, or $0$ if $\nu=((\mu_1^t-1, \mu_2^t, \cdots, \mu_{n-k+1}^t)//\mu^b)$.

 Hence, the statement follows.
 \end{proof}
\begin{thm}
[Quantum Pieri rule for $OG(k, 2n+1)$]\label{thmQPRforOGOG222}  Let $
\mathbf{a}\in\mathcal{P}_{k}$ and $p\leq k$. Using the same notations as in Lemma
\ref{lemmaforPieriidenti}, we have
\[
\sigma^{p}\star\sigma^{\mathbf{a}}=\sum2^{e'(\mathbf{a}, \mathbf{b})}%
\sigma^{\mathbf{b}} + t\sum N_{p, \mathbf{a}}^{\mathbf{c}, 1}%
\sigma^{\mathbf{c}}+t^2\sum2^{e'(\hat{\mathbf{a}}, \hat{\mathbf{d}})}%
\sigma^{\mathbf{d}}.
\]
Here the first summation is over all shapes $\mathbf{b}\in\mathcal{P}_{k}$
with $|\mathbf{b}|=|\mathbf{a}|+p$ such that $\mathbf{b}$ is compatible with
$\mathbf{a}$;   the third summation occurs only if $\hat{\mathbf{a}}%
\in\mathcal{P}_{k}$, and when this holds, the summation is over all shapes $\mathbf{d}%
\in\mathcal{P}_{k}$ with $|\mathbf{d}|=|\mathbf{a}|+p-4n+2k$ such that
$\hat{\mathbf{d}}\in\mathcal{P}_{k}$ is compatible with $\hat{\mathbf{a}}%
$;  the second summation is over all shapes $\mathbf{c}%
\in\mathcal{P}_{k}$ with $|\mathbf{c}|=|\mathbf{a}|+p-2n+k$. Furthermore,  we have
  $$ N_{p, \mathbf{a}}^{\mathbf{c}, 1}=\begin{cases}
        2^{e'(\bar{\mathbf{a}}, \tilde{\mathbf{c}})}, &\mbox{if } a_1^t>a_1^b\geq a_2^t\\
         2^{e'(\tilde{\mathbf{a}}, \bar{\mathbf{c}})}, & \mbox{if } a_1^b\geq a_1^t, c_1^t=n \mbox{ and } c_2^t\leq n-2\\
          M, & \mbox{if } a_1^b\geq a_1^t    \mbox{ and } c_1^t<n\\
          0, & \mbox{otherwise}
  \end{cases}  .$$
with
  $$M=\sum_\mu 2^{e'(\tilde{\mathbf{a}}, \mu)}+\sum_\mu 2^{1+e'(\tilde{\mathbf{a}}, \mu)}-\sum_\nu 2^{e'(\nu, \tilde{\mathbf{c}})}-\sum_\nu 2^{1+e'(\nu, \tilde{\mathbf{c}})}$$
where the first sum is over
   $\{\mu~|~ \mu\in \Gamma_1(\tilde{\mathbf{c}})\}$, the second sum is over
      $\{\mu~|~ \mu\in \Gamma_2(\tilde{\mathbf{c}})\}$,
       the third sum is over
      $\{\nu~|~ \tilde{\mathbf{a}}\in \Gamma_1(\nu)\}$,
    and   the last sum is over
      $\{\nu~|~ \tilde{\mathbf{a}}\in \Gamma_2(\nu)\}$.
\end{thm}
\begin{proof}
  The first summation is provided from the classical Pieri rule.   Note  $\deg t=2n-k$ for $OG(k, 2n+1)$.
   The third summation follows directly from
   Theorem \ref{thmQPRforOGOG}, Lemma \ref{lemmaforthmQPROG}, (statement (2c), (4) of) Lemma \ref{lemmaforPieriidenti} and Proposition \ref{propCPRBBB}.

   Because of the dimension constraint,   $N_{p, \mathbf{a}}^{\mathbf{c}, 1}=N_{u, v}^{w, 1}=N_{us_k, v_1}^{w_1, 0}$ is nonzero only if
    $\ell(v_1)=\ell(v)-2n+2k$ (see Lemma \ref{lemmaforPieriidenti} for the notations).
   When this holds, we have $a_1^b\geq a_2^t$. Note $a_1^t>a_2^t$.

   Assume  $a_1^t>a_1^b$; then $v_1\in W^{\tilde P}$, so that $N_{us_k, v_1}^{w_1, 0}\neq 0$ only if $w_1\in W^{\tilde P}$. Thus in this case, we have
     $N_{p, \mathbf{a}}^{\mathbf{c}, 1} = 2^{e'(\bar{\mathbf{a}}, \tilde{\mathbf{c}})}$ by using statement (2b) and (3a) of Lemma \ref{lemmaforPieriidenti}
        together with the classical Pieri rule for $OG(k-1, 2n+1)$.

   Assume  $a_1^b\geq a_1^t$; then $v_1=vs_\gamma s_k$ with $vs_\gamma\in W^{\tilde P}$, so that $\mbox{sgn}_k(v_1)=1$.
  \begin{enumerate}
     \item Suppose $w_1\notin W^{\tilde P}$; then by Corollary \ref{coridentity}, we  conclude that
             $N_{us_k, v_1}^{w_1, 0}\neq 0$ only if $\mbox{sgn}_k(w_1)=1$, and
             $N_{us_k, v_1}^{w_1, 0}=N_{us_k, vs_\gamma}^{w_1s_k, 0}$ when this holds. In particular, we have
      $w_1s_k\in W^{\tilde P}$. Thus if
        $c_1^t=n$   and   $c_2^t\leq n-2$,
   then we have
     $N_{p, \mathbf{a}}^{\mathbf{c}, 1} =
          2^{e'(\tilde{\mathbf{a}}, \bar{\mathbf{c}})}$, by using  Lemma \ref{lemmaforPieriidenti} and Proposition \ref{propCPRBBB}.
     \item Suppose $w_1\in W^{\tilde P}$; that is, $c_1^t<n$. Using Corollary \ref{coridentity} again, we conclude
                         $\sigma^{s_k}\cup \sigma^{vs_\gamma}=\sigma^{vs_\gamma s_k}+\sum_{v'\in W^{\tilde P}} N_{s_k, vs_\gamma}^{v', 0}\sigma^{v'}$.
         Hence, $N_{us_k, v_1}^{w_1, 0}=              N_{us_k, vs_\gamma s_k}^{w_1, 0}=\sum_{w'\in W^{\tilde P}}N_{s_k, w'}^{w_1, 0}N_{us_k, vs_\gamma}^{w', 0}
             -\sum_{v'\in W^{\tilde P}}N_{s_k, vs_\gamma}^{v', 0}N_{us_k, v'}^{w_1, 0}$.
          Thus in this case, the statement follows from Lemma \ref{lemmafordeg1ChevallBBB}, Lemma \ref{lemmaforPieriidenti} and Proposition \ref{propCPRBBB}.
  \end{enumerate}
\end{proof}

\begin{remark}
  One might want to use the method of Buch-Kresch-Tamvakis on page 372 of \cite{BKT2} to compute $N_{p, \mathbf{a}}^{\mathbf{c}, 1}$, but
    replace $OF(m, m+1; 2n+1)$ (resp. $OG(m+1, 2n+1)$) by $OF(m-1, m; 2n+1)$ (resp. $OG(m-1, 2n+1)$).
  However, such a method does not work here, since
     the identity ${\varphi_2}_*\pi_2^*{\pi_2}_* \varphi_2^*\varphi_1^*\tau_\mu=  \pi^*\pi_*  \varphi_1^*\tau_\mu$ 
     used in their proof   does not hold  any more in our case.
\end{remark}
\subsection{Remarks}
\label{subsecremarks}
As in the previous two subsections, we have derived the quantum Pieri rules for Grassmannians of type $C_n$ and $B_n$ for $k<n$.
The remaining cases for these two types are about the Lagrangian Grassmannian $IG(n, 2n)$ and the odd orthogonal Grassmannian $OG(n, 2n+1)$.
Theorem \ref{thmQPRforIGIG} also works for $IG(n, 2n)$, therefore we can also reformulate it in the way of  Theorem \ref{thmQPRforIGIG222}, by writing down
   the  explicit  identifications as needed and using the classical Pieri rule of Hiller and Boe \cite{HillerBoe}.
For $OG(n, 2n+1)$,  we can also reformulate Theorem \ref{thmQPRforOGOG} in a nice way like Theorem \ref{thmQPRforIGIG222}, by using
generalized classical Pieri rules (Theorem D) given by Bergeron and Sottile
\cite{bergSottile}.  The former case has been done by Kresch and Tamvakis in \cite{KreschTamvakisLagran}, and the latter case has also been done by them
 in an equivalent way in  \cite{KreschTamvakisortho}. Hence, we skip the details here.

When $\Delta$ is of type $D_{n+1}$, the minimal length representatives $W^P$ for the isotropic Grassmannian $OG(k, 2n+2)$ can also be identified with
  ``shapes", which consist in  two types. There are parallel propositions to Lemma \ref{lemmaforPieriidenti}, Lemma \ref{lemmafordeg1ChevallBBB} and
   Theorem \ref{thmQPRforOGOG222} with similar arguments. Since the notations are more involved and the theorem as expected may also involve signs in some cases, we
   skip the details.

 Assuming the classical Pieri rules with respect to the Chern classes $c_p(\mathcal{Q})$ of the tautological quotient bundle $\mathcal{Q}$ over the isotropic Grassmannians, we can also reprove the quantum Pieri rule of Buch-Kresch-Tamvakis for type $B_n$ and $D_{n+1}$.
   For instance for $G/P=OG(k, 2n)$ where $k<n$,  the Chern classes of the tautological quotient bundle $\mathcal{Q}$
are given by $\sigma^{u}$ (up to a scale factor of $2$) for $u$ of the following form (see section 6 of
\cite{BKT3})
\[
s_{k}, s_{k+1}s_{k}, \cdots, s_{n}s_{n-1}\cdots s_{k}, s_{n-1}s_{n}%
s_{n-1}\cdots s_{k}, \cdots, s_{1}\cdots s_{n-1}s_{n}s_{n-1}\cdots s_{k}.
\]
Taking any one of the above $u$ and any $v\in W^{P}$, we can show the following quantum Pieri rule given by Buch, Kresch and Tamvakis.

\begin{prop}[See Theorem 2.4 of \cite{BKT2}]
   \[
\sigma^{u}\star\sigma^{v}=\sigma^{u}\cup\sigma^{v} +  \sum_{w\in W^{P}\atop \ell(u)+\ell(v)=\ell(w)+2n-k}
 N_{us_k, v_{1}}^{w_{1}, 0} \sigma^{w}t+  \sum_{w\in W^{P}\atop \ell(u)+\ell(v)=\ell(w)+4n-2k}
N_{u, v_2}^{w_2, 0} \sigma^{w}t^2,
\]
in which we have $v_{1}=vs_ks_{k-1}\cdots s_1$, $w_1=ws_{k+1}\cdots s_{n-1}s_ns_{n-1}\cdots s_{k+1}$,
  $ v_{2}%
=vs_{k}\cdots s_{n-1}s_{\bar\alpha_{n}} s_{n-1}\cdots s_1$
and  $w_{2}=ws_{1}\cdots s_{n-1} s_{\bar\alpha_{n}}%
s_{n-1}\cdots s_{k}$.
 \end{prop}
 We  sketch    an alternative proof of the above proposition as follows.
  \begin{enumerate}
    \item Using similar arguments for Proposition \ref{propdlarger2BBDD}, we can show that $N_{u, v}^{w, d}=0$ if $d\geq 3$.
    \item Using Proposition \ref{propmainthminLeungLi}, we can show $N_{u, v}^{w, 1}=N_{us_k, v_1}^{w_1, 0}$ and
      $N_{u, v}^{w, 2}=N_{u, v_2}^{w_2, 0}$. (The arguments here become much simpler than the proof of Theorem \ref{thmQPRforOGOG}.)
  \end{enumerate}
Furthermore for each $i\in \{1, 2\}$, it follows directly from the dimension constraint that
   $N_{u, v_i}^{w_i, 0}\neq 0 \mbox{ only if } \ell(v_i)=\ell(v)-\ell(v^{-1}v_i)$. When this holds, we have
 \begin{enumerate}
   \item[(3)] $v_1\in W^{\tilde P}$ where $G/\tilde P=OG(k+1, 2n+1)$. Furthermore, let $\lambda$ denote the $(n-k)$-strict partition corresponding to $v$; then
           $v_1$  exactly corresponds the $(n-k-1)$-strict partition $\bar \lambda$ defined on page 372 of \cite{BKT2}.
   \item[(4)]  Statement (2c) and (4) of Lemma \ref{lemmaforPieriidenti}, which give  identifications between $W^P$ and shapes (and therefore other parameterizations), can be applied directly for $v_2$ and $w_2$ respectively.
  \end{enumerate}

 For $IG(k, 2n)$ (i.e. Grassmannians of type $C_n$), we can also show that there are most degree 0 and  1 Gromov-Witten invariants  occurring  in the quantum Pieri rule; furthermore,  we can reduce the degree 1 Gromov-Witten invariants to certain classical intersection numbers,  for $1\leq p\leq  n-k+1$.

We  can also use  Proposition \ref{propmainthminLeungLi} to compute $N_{u, v}^{w, d}$ for certain $u, v, w, d$, in which none of
   $u, v\in W^P$ is special.

\begin{example}
 For $X=IG(3, 10)$, we take $w=s_2$, $u=s_2s_1s_4s_3s_2s_5s_4s_3$ and $v=s_1s_5s_4s_3s_2s_4s_5s_4s_3$. By using
  Proposition \ref{propmainthminLeungLi} directly, we can show
   $N_{u, v}^{w, 2}=N_{u, v'}^{w', 0}$ where
   $v'=s_1s_3$ and $w'=s_3s_2s_1s_5s_4s_3s_2s_5s_4s_3$.
  We can easily  compute the
   classical intersection number $N_{u, v'}^{w', 0}$, for instance by using the   Chevalley formula with the observation that
      $\sigma^{s_1s_3}=\sigma^{s_1}\cup \sigma^{s_3}$.
    As a consequence, we have      $N_{u, v}^{w, 2}=N_{u, v'}^{w', 0}=1$.
   (Note that in terms of the notations in  Example 1.5 of \cite{BKT2},
           we have $\sigma^u=\sigma_{4, 2,, 2}, \sigma^v=\sigma_{5, 3, 1}$ and $\sigma^{w}=\sigma_1$.)
\end{example}

\section{Appendix} 

In this appendix, we   reprove the well-known quantum Pieri rules for Grassmannians of type $A$ (i.e. complex Grassmannians) with the same method used in
the present paper.

Recall that   $Gr(k, n+1)=SL(n+1, \mathbb{C})/P$ with   $P$ being
 the (standard) maximal parabolic subgroup  corresponding     $\Delta_{P}=\Delta\setminus\{\alpha_{k}\}$.
The Weyl group $W$ is canonically isomorphic to  the permutation group $S_{n+1}$, by mapping $s_j$ to the transposition $(j, j+1)$ for each $1\leq j\leq n$.
Customarily, Schubert classes in $H^*(Gr(k, n+1))$ are labelled by the set $$\mathcal{P}_{k}
  =\{\mathbf{a}=(a_1, \cdots, a_k)~|~ n+1-k\geq a_1\geq a_2\geq \cdots \geq a_k\geq 0\}$$ of partitions inside the
   $k$ by $(n+1-k)$ rectangle. Precisely, for each $u\in W^P$, we can rewrite the Schubert cohomology class
    $\sigma^u$ as $\sigma^{\mathbf{a}(u)}$ with
       $$  \mathbf{a}(u)=(u(k)-k, u(k-1)-(k-1), \cdots, u(2)-2, u(1)-1)$$
  via the canonical identification between $W^P$ and $\mathcal{P}_k$ (see e.g. \cite{fw}).
 For the   tautological bundles, $0\rightarrow \mathcal{S}\rightarrow \mathbb{C}^{n+1}\rightarrow \mathcal{Q}\rightarrow 0$,  the $i$-th Chern class  $c_i(\mathcal{Q})$ (resp. $c_i(\mathcal{S}^*)$)
  coincides with
  the Schubert cohomology  class $\sigma^{\mathbf{a}(u)}=\sigma^u$  where
   $u=s_{k+i-1}s_{k+i-2} \cdots s_{k+1}s_{k}$ and consequently $\mathbf{a}(u)=(i, 0, \cdots,0 )=:i$ for each $1\leq i\leq n+1-k$
   (resp.  $u=  s_{k-i+1}s_{k-i+2}\cdots s_{k-1}s_k$
     and consequently $\mathbf{a}(u)=(1^i)$ for each $1\leq i\leq k$) (see e.g. \cite{buch}).
The following   quantum Pieri rule with respect to $c_p(\mathcal{Q})$ was  first proved by Bertram \cite{bert}.


\begin{prop}[Quantum Pieri rule]\label{propquanpiericpxgrassm}
  For any  $p$,   $\mathbf{a}\in \mathcal{P}_k$, we have
    $$\sigma^p\star \sigma^{\mathbf{a}}=\sum\sigma^{\mathbf{b}}+t\sum\sigma^{\mathbf{c}},$$
where the first is summation over all $\mathbf{b}\in\mathcal{P}_k$ satisfying
    $|\mathbf{b}|=|\mathbf{a}|+p$ and $n+1-k\geq b_1\geq a_1\geq b_2\geq a_2\geq\cdots \geq b_k\geq a_k$, and
 the second summation is over all   $\mathbf{c}\in\mathcal{P}$ satisfying
    $|\mathbf{c}|=|\mathbf{a}|+p-(n+1)$ and $a_1-1\geq c_1\geq a_2-1\geq c_2\geq  \cdots \geq a_k-1\geq c_k$.
\end{prop}

\noindent We will reprove the above proposition, rather than a quantum Pieri rule with respect to $c_p(\mathcal{S}^*)$, for the following two reasons.

\begin{enumerate}
  \item The tautological quotient bundle over $Gr(k, n+1)=Gr(k, \mathbb{C}^{n+1})$ is isomorphic to
       the dual of the tautological subbundle over the dual Grassmannian $Gr(n+1-k, (\mathbb{C}^{n+1})^*)\cong Gr(n+1-k, n+1)$.
        Hence, it is sufficient to know quantum Pieri rules with respect to either $c_p(\mathcal{Q})$ or $c_p(\mathcal{S}^*)$ for all complex Grassmannians.
  \item From   our alternative proof, we can see that one   point  of our method relies on the combinatorial property of $u$, 
  rather than the geometric property (of being a Chern class of $\mathcal{Q}$ or $\mathcal{S}^*$) for $\sigma^u$.
        In particular, our method could have more general applications (see e.g. Theorem 1.2 of \cite{leungliQtoC}).
\end{enumerate}

\noindent We will use the classical Pieri rule (as given by the first summation in the above formula) and the next two lemmas, which are easily deduced
 from Proposition \ref{propmainthminLeungLi}.
In addition, we use all the same notations as in section \ref{subsectisotroGrassm} but set
   $u=s_{k+p-1}s_{k+p-2}\cdots s_k$ (where   $1\leq p\leq n+1-k$) and $v\in W^P$ with $\mathbf{a}=\mathbf{a}(v)$.
 We need to compute the Gromov-Witten invariants $N_{u, v}^{w, d}$ for the quantum product $\sigma^u\star \sigma^v$.
\begin{lemma}\label{lemmavanishcpxgrass}
 If $d\geq 2$, then we have    $N_{u, v}^{w, d}=0$ for any $w\in W^P$.
\end{lemma}

\begin{proof}
 Recall that  $N_{u, v}^{w, d}=N_{u, v}^{w, \lambda_P}=N_{u, v}^{w\omega_P\omega_{P'}, \lambda_B}$ where $\lambda_P=d\alpha_k^\vee+Q^\vee_P$.
Furthermore, we denote
  $d=m_1k+r_1=m_2(n-k+1)+r_2$ where $1\leq r_1\leq k$ and $1\leq r_2\leq n-k+1$.
  Then we conclude $\lambda_B=m_1\sum_{j=1}^{k-1}j\alpha_{j}^{\vee}+\sum_{j=1}^{r_1-1}%
j\alpha_{k-r_1+j}^{\vee}+ d\alpha_k^\vee+ m_2\sum_{j=1}^{n-k}j\alpha_{n+1-j}^{\vee}+\sum_{j=1}^{r_2-1}j\alpha_{k+r_2-j}^{\vee}$
    (by noting $\langle \alpha_i, \lambda_B\rangle=-1$ if $i\in \{k-r_1, k+r_2\}$, or $0$ otherwise). When $k=1$ (resp. $n$),
    then $\langle \alpha_k, \lambda_B\rangle=d+m_2+1$ (resp. $d+m_1+1$) is larger than $2$
     and thus  we are   done by Proposition \ref{propmainthminLeungLi} (1).
  When $2\leq k\leq n-1$, we have  $\langle \alpha_k, \lambda_B\rangle=m_1+m_2+2$.
   We still have $N_{u, v}^{w, d}=0$ unless
     $m_1=m_2=0$
              and $\mbox{sgn}_k(w\omega_P\omega_{P'})=0$.
  When both conditions hold,  we have $N_{u, v}^{w\omega_P\omega_{P'}, \lambda_B}=
   N_{us_k, v}^{w\omega_P\omega_{P'} s_k, \lambda_B-\alpha_k^\vee} $ by Proposition \ref{propmainthminLeungLi} (2).
  Furthermore, we have
    $\mbox{sgn}_{k-1}(us_k)=\mbox{sgn}_{k-1}(v)=0<1= \langle \alpha_{k-1}, \lambda_B-\alpha_k^\vee\rangle$. Hence, we have
       $N_{us_k, v}^{w\omega_P\omega_{P'} s_k, \lambda_B-\alpha_k^\vee}=0$ by
     using Proposition \ref{propmainthminLeungLi} (1) again.
\end{proof}

Note that $\ell(u)=p$ and the degree of the quantum variable is   $\deg t=n+1$.

\begin{lemma}\label{lemmadegree1pieri}
 For any  $w\in W^P$ with $\ell(w)=\ell(v)+p-(n+1)$, we have
    $$ N_{u, v}^{w,  1}=N_{us_k, vs_ks_{k-1}\cdots s_2s_1 }^{ws_{n}s_{n-1}\cdots s_{k+1}, 0}.$$
\end{lemma}
\begin{proof}
  Note that for $\lambda_P= \alpha_k^\vee+Q^\vee_P$, we have $\lambda_B=\alpha_k^\vee$ and consequently $\Delta_{P'}=\Delta_P\setminus\{\alpha_{k-1},   \alpha_{k+1}\}$.
Then conclude $\omega_P\omega_{P'}=s_{n}s_{n-1}\cdots s_{k+1}s_{1}s_{2}\cdots s_{k-1}$ by checking the assumptions in Lemma \ref{charalongest} directly.
Thus we have  $ N_{u, v}^{w,  1}= N_{u, v}^{w\omega_P\omega_{P'}, \alpha_k^\vee}=
                N_{us_k, vs_k}^{w\omega_P\omega_{P'},0}$, by using Proposition \ref{propmainthminLeungLi} (2).

Note that $\mbox{sgn}_{k-1}(us_k)=0$ and $\mbox{sgn}_{k-1}(w\omega_P\omega_{P'})=1$.
By Corollary \ref{coridentity}, we have $N_{us_k, vs_k}^{w\omega_P\omega_{P'}, 0}=N_{us_k, vs_ks_{k-1}}^{w\omega_P\omega_{P'}s_{k-1}, 0}$ if
 $\ell(vs_ks_{k-1})=\ell(vs_k)-1$, or $0$ otherwise. By induction and using   Corollary \ref{coridentity} repeatedly, we conclude
  $N_{us_k, vs_k}^{w\omega_P\omega_{P'}, 0}=N_{us_k, vs_ks_{k-1}\cdots s_1}^{w\omega_P\omega_{P'}s_{k-1}\cdots s_1, 0}$ if
 $\ell(vs_ks_{k-1}\cdots s_1)=\ell(vs_k)-(k-1)$, or $0$ otherwise.
 Furthermore, we note $\ell(us_k)+\ell(vs_k)=\ell(w\omega_P\omega_{P'})=\ell(w\omega_P\omega_{P'}s_{k-1}\cdots s_1)+(k-1)$.
 Thus if $\ell(vs_ks_{k-1}\cdots s_1)\neq \ell(vs_k)-(k-1)$, then we have
   $N_{us_k, vs_ks_{k-1}\cdots s_1}^{w\omega_P\omega_{P'}s_{k-1}\cdots s_1, 0}=0$, due to the dimension constraint. Hence, the statement follows.
\end{proof}

It remains to apply the classical Pieri rule and reformulate the product $\sigma^u\star \sigma^v$ in terms of $\sigma^p\star \sigma^{\mathbf{a}}$ (labbelld by partitions).

\bigskip
\begin{proof}[Proof of Proposition \ref{propquanpiericpxgrassm}]
Denote   $\mathbf{c}=\mathbf{c}(w)\in \mathcal{P}_{k}$.
It follows directly from Proposition \ref{propmainthminLeungLi} (1) and    the dimension constraint that
     $N_{us_k, vs_ks_{k-1}\cdots s_2s_1 }^{ws_{n}s_{n-1}\cdots s_{k+1}, 0}=0 $ unless
     $\tilde{\mathbf{a}}=\tilde{\mathbf{a}}(vs_ks_{k-1}\cdots s_2s_1)$ and
     $\tilde{\mathbf{c}}=\tilde{\mathbf{c}}(ws_ns_{n-1}\cdots s_{k+1})$ are both partitions in
     $\mathcal{P}_{k+1}$ for $Gr(k+1, n+1)$.
   As a consequence, we have that
    $\tilde a_{i}=vs_ks_{k-1}\cdots s_2s_1(k+2-i)-(k+2-i)=v(k+1-i)-(k+1-i)-1=a_i-1$ for each $1\leq i\leq k$,
     and that  $\tilde c_{j}=ws_ns_{n-1}\cdots s_{k+1}(k+2-j)-(k+2-j)=w(k+2-j)-(k+2-j)=c_{j-1}$ for each $2\leq j\leq k+1$.
 Note that $|\tilde{\mathbf{c}}|=\ell(ws_ns_{n-1}\cdots s_{k+1})=\ell(w)+n-k=|\mathbf{c}|+n-k$. We have $\tilde c_1=n-k$.
 For  $|\tilde{\mathbf{a}}|=\ell(vs_ks_{k-1}\cdots s_{1})=\ell(v)-k=|\mathbf{a}|-k$, we conclude $\tilde{a}_{k+1}=0$.
  In addition, we note that ${us_k}$ corresponds to the special partition
    $p-1$ in $\mathcal{P}_{k+1}$ for $Gr(k+1, n+1)$.   Hence,  the non-classical part of the quantum product $\sigma^p\star \sigma^{\mathbf{a}}$
      is exactly the second half as in the statement,   by using the classical Pieri rule for $\sigma^{p-1}\cup \sigma^{\tilde{\mathbf{a}}}$
      in $H^*(Gr(k+1, n+1))$.
\end{proof}

\begin{remark}
 This is the same approach as taken in the  elementary proof of the quantum Pieri rule by Buch  \cite{buch},
  where he uses a different method to obtain   Lemma \ref{lemmavanishcpxgrass} and Lemma \ref{lemmadegree1pieri}.
\end{remark}

\section{Acknowledgements}

The authors thank Leonardo Constantin Mihalcea and  Harry Tamvakis for their generous helps, and thank the referee for his/her valuable comments and suggestions. We also thank
Ionut Ciocan-Fontanine, Baohua Fu, Bumsig Kim, Thomas Lam
and Frank Sottile and for useful
discussions.  The first author is supported in part by a RGC research grant
from the Hong Kong government.  The second author is supported in part by  JSPS Grant-in-Aid for Young Scientists (B) 25870175. The second author is   grateful for
Korea Institute for Advanced Study, where he worked and this project started.


\begin{thebibliography}{99}                                                                                               %


\bibitem{bergSottile}N. Bergeron, F. Sottile,\,\textit{A Pieri-type formula
for isotropic flag manifolds}, Trans. Amer. Math. Soc. 354 (2002), 2659--2705.

\bibitem{bert}A. Bertram,\,\textit{Quantum Schubert calculus}, Adv. Math. 128
(1997), no. 2, 289--305.





\bibitem{buch}A.S. Buch,\,\textit{Quantum cohomology of Grassmannians},
Compositio Math. 137 (2003), no. 2, 227--235.

\bibitem{BKT1}A.S. Buch, A. Kresch and H. Tamvakis,\,\textit{Gromov-Witten
invariants on Grassmannians}, J. Amer. Math. Soc. 16 (2003), no. 4, 901--915 (electronic).

\bibitem{BKT2}A.S. Buch, A. Kresch and H. Tamvakis,\,\textit{Quantum Pieri
rules for isotropic Grassmannians}, Invent. Math. 178 (2009), no. 2, 345--405.

\bibitem{BKT3}A.S. Buch, A. Kresch and H. Tamvakis,\,\textit{A Giambelli
formula for isotropic Grassmannians}, arXiv: math.AG/0811.2781.

\bibitem{buchMihalcea}A.S. Buch, L.C. Mihalcea,\,\textit{Quantum K-theory of
Grassmannians}, Duke Math. J. 156 (2011), no. 3, 501--538.

\bibitem{buchMihalcea22}A.S. Buch, L.C. Mihalcea,\,{\it Curve neighborhoods and the quantum Chevalley formula},  preprint available at http://www.math.rutgers.edu/$\scriptstyle\sim$asbuch/papers.


\bibitem{cmn}P.E. Chaput, L. Manivel and N. Perrin,\,\textit{Quantum cohomology
of minuscule homogeneous spaces}, Transform. Groups 13 (2008), no. 1, 47--89.

 \bibitem{chper}P.-E. Chaput, N. Perrin,\,{\it  On the quantum cohomology of adjoint varieties}, Proc. Lond. Math. Soc. (3) 103 (2011), no. 2, 294--330.


\bibitem{CFon}I. Ciocan-Fontanine,\,\textit{On quantum cohomology rings of
partial flag varieties}, Duke Math. J. 98 (1999), no. 3, 485--524.

 \bibitem{coskun33}I. Coskun,\,{\it The quantum cohomology of flag varieties and the periodicity of the Schubert structure constants}, Math. Ann. 346 (2010), no. 2, 419--447.








\bibitem{fupa}W. Fulton, R. Pandharipande,\,\textit{Notes on stable maps and
quantum cohomology}, Proc. Sympos. Pure Math. 62, Part 2, Amer. Math. Soc.,
Providence, RI, 1997.

\bibitem{fw}W. Fulton, C. Woodward,\,\textit{On the quantum product of
Schubert classes}, J. Algebraic Geom. 13 (2004), no. 4, 641--661.

\bibitem{gri}P. Griffiths, J. Harris,\,\textit{Principles of algebraic
geometry}, John Wiley $\&$ Sons, Inc., New York, 1994.



\bibitem{hiller}H. Hiller,\,\textit{Combinatorics and intersections of
Schubert varieties}, Comment. Math. Helv. 57 (1982), no. 1, 41--59.

\bibitem{HillerBoe}H. Hiller, B. Boe,\,\textit{Pieri formula for
$SO_{2n+1}/U_{n}$ and $Sp_{n}/U_{n}$}, Adv. in Math. 62 (1986), no. 1, 49--67.

\bibitem{hum}J.E. Humphreys,\,\textit{Introduction to Lie algebras and
representation theory}, Graduate Texts in Mathematics 9, Springer-Verlag, New
York-Berlin, 1980.


\bibitem{humr}J.E. Humphreys,\,\textit{Reflection groups and Coxeter groups},
Cambridge University Press, Cambridge, UK, 1990.





\bibitem{KreschTamvakisLagran}A. Kresch, H. Tamvakis,\,\textit{Quantum
cohomology of the Lagrangian Grassmannian}, J. Algebraic Geom. 12 (2003), no.
4, 777--810.

\bibitem{KreschTamvakisortho}A. Kresch, H. Tamvakis,\,\textit{Quantum
cohomology of orthogonal Grassmannians}, Compos. Math. 140 (2004), no. 2, 482--500.

\bibitem{lamshi}T. Lam, M. Shimozono,\,\textit{Quantum cohomology of $G/P$
and homology of affine Grassmannian}, Acta Math. 204 (2010), no. 1, 49--90.

\bibitem{LasSch}A. Lascoux, M.-P. Sch\"utzenberger,\,\textit{Polyn\^{o}mes de
Schubert}, C. R. Acad. Sci. Paris S\'er. I Math. 294 (1982), no. 13, 447--450.



\bibitem{leungli33}N.C. Leung, C. Li,\,\textit{Functorial relationships
between $QH^{*}(G/B)$ and $QH^{*}(G/P)$},  J. Differential Geom. 86 (2010), no. 2, 303--354.

\bibitem{leungliQtoC}N.C. Leung, C. Li, \,\textit{Classical aspects of
quantum cohomology of generalized flag varieties}, Int. Math. Res. Not. IMRN 2012, no. 16, 3706--3722.



\bibitem{liMihal}C. Li, L.C. Mihalcea,\,\textit{K-theoretic Gromov-Witten invariants of lines in homogeneous spaces}, Int. Math. Res. Not. 2013, Art. ID rnt090; preprint at arxiv: math.AG/1206.3593.








\bibitem{pragrat}P. Pragacz, J. Ratajski,\,\textit{A Pieri-type theorem for
Lagrangian and odd orthogonal Grassmannians}, J. Reine Angew. Math. 476
(1996), 143--189.

\bibitem{pragratQpoly}P. Pragacz, J. Ratajski,\,\textit{Formulas for
Lagrangian and orthogonal degeneracy loci; $\tilde Q$-polynomial approach},
Compositio Math. 107 (1997), no. 1, 11--87.

\bibitem{pragratEvenorth}P. Pragacz, J. Ratajski,\,\textit{A Pieri-type
formula for even orthogonal Grassmannians}, Fund. Math. 178 (2003), no. 1, 49--96.

\bibitem{peterson}D. Peterson,\,\textit{Quantum cohomology of $G/P$}, Lecture
notes at MIT, 1997 (notes by J. Lu and K. Rietsch).



\bibitem{Sertoz}S. Sert\"oz,\,\textit{A triple intersection theorem for the
varieties $SO(n)/P_{d}$}, Fund. Math. 142 (1993), no. 3, 201--220.

\bibitem{sietian}B. Siebert, G. Tian,\,\textit{On quantum cohomology rings of
Fano manifolds and a formula of Vafa and Intriligator},  Asian J. Math. 1
(1997), no. 4, 679--695.



\bibitem{tamvakis}H. Tamvakis,\,\textit{Quantum cohomology of isotropic
Grassmannians}, Geometric methods in algebra and number theory, 311--338,
Progr. Math., 235, Birkh\"auser Boston, Boston, MA, 2005.

\bibitem{wo}C.T. Woodward,\,\textit{On D. Peterson's comparison formula for
Gromov-Witten invariants of $G/P$},  Proc. Amer. Math. Soc. 133 (2005), no. 6,
1601--1609.

\end{thebibliography}
\end{document}